\documentclass[reqno]{amsart}
\usepackage{amsmath,amsthm,amscd,amssymb,amsfonts, amsbsy}
\usepackage{latexsym, color, enumerate}
\usepackage{pxfonts}

\theoremstyle{plain}
\newtheorem{theorem}[equation]{Theorem}
\newtheorem{lemma}[equation]{Lemma}

\newtheorem{proposition}[equation]{Proposition}

\theoremstyle{definition}

\theoremstyle{remark}
\newtheorem{remark}[equation]{Remark}

\newcommand{\I}{\operatorname{I}}

\numberwithin{equation}{section}

\newcommand{\bA}{\mathbf{A}}
\newcommand{\bF}{\mathbf{F}}

\newcommand{\bN}{\mathbb{N}}

\newcommand{\bR}{\mathbb{R}}

\newcommand\cG{\mathcal{G}}

\providecommand{\norm}[1]{\lVert#1\rVert}

\newcommand{\wei}[1]{\langle #1 \rangle}
\newcommand{\A}{\mathbf{A}}
\newcommand{\F}{\mathbf{F}}
\newcommand{\G}{\mathbf{G}}

\begin{document}

\title[Parabolic equations with singular degenerate coefficients]{Regularity  theory for parabolic equations with\\ singular degenerate coefficients}

\author[H. Dong]{Hongjie Dong}
\address[H. Dong]{Division of Applied Mathematics, Brown University,
182 George Street, Providence, RI 02912, United States of America}
\email{Hongjie\_Dong@brown.edu}

\author[T. Phan]{Tuoc Phan}
\address[T. Phan]{Department of Mathematics, University of Tennessee, 227 Ayres Hall,
1403 Circle Drive, Knoxville, TN 37996-1320 }
\email{phan@math.utk.edu}

\thanks{H. Dong was partially supported by the NSF under agreement  DMS-1600593. T. Phan is partially supported by the Simons Foundation, grant \# 354889.}

\begin{abstract}
In this paper, we study parabolic equations in divergence form with coefficients that are singular degenerate as some Muckenhoupt weight functions in one spatial variable. Under certain conditions, weighted reverse H\"{o}lder's inequalities are established. Lipschitz estimates for weak solutions are proved for homogeneous equations with singular degenerate coefficients depending only on one spatial variable.  These estimates are then used to establish interior, boundary, and global weighted estimates of Calder\'{o}n-Zygmund type for weak solutions, assuming that the coefficients are partially \textup{VMO} (vanishing mean oscillations) with respect to the considered weights. The solvability in weighted Sobolev spaces is also achieved. Our results are new even for elliptic equations, and non-trivially extend known results for uniformly elliptic and parabolic equations. The results are also useful in the study of fractional elliptic and parabolic equations with measurable coefficients.
\end{abstract}

\maketitle

\section{Introduction and main results}

Consider a given weight function $\mu : \bR \rightarrow (0, \infty)$ that could be singular or degenerate. In this paper, we establish existence, uniqueness, and regularity theory in weighted Sobolev spaces for solutions of the following class of parabolic equations with coefficients that are measurable and singular-degenerate in $x_n$-variable
\begin{equation} \label{eqn-entire}
\mu(x_n) a_0(x_n)u_t - \textup{div}[\mu(x_n) (\mathbf{A}  \nabla u - \F)] +\lambda\mu(x_n)  u = \mu(x_n)f
\end{equation}
in  $(-\infty, T) \times \bR^n$, where $T\in (-\infty,+\infty]$.
We also study the class of equations \eqref{eqn-entire} in the half space $(-\infty,T)\times \bR^n_+$
with the conormal boundary condition:
\begin{equation} \label{main-eqn}
\lim_{x_n \rightarrow 0^+}\wei{\mu(x_n)(\A \nabla u - \F), \mathbf{e}_n }  = 0.
\end{equation}
Here, 
$\bR^n_+ =\bR^{n-1} \times \bR_{+}$ with some $n \in \mathbb{N}$, $\mathbf{e}_n = (0, 0, \ldots, 0, 1) \in \bR^n$ is the unit $n^{\textup{th}}$-coordinate vector, $\lambda\ge 0$, and $\A$ is a given measurable bounded and uniformly elliptic matrix, $a_0$ is a given measurable function, $\F$ is a given measurable vector field, and $f$ is a given measurable function. Throughout the paper, we assume that  there exists $\Lambda >0$ such that
\begin{equation} \label{ellipticity}
\begin{split}
& \Lambda^{-1} \leq a_0(x_n) \leq \Lambda  \quad \text{for a.e.} \, x_n, \quad \text{and} \\
& \Lambda^{-1} |\xi|^ 2 \leq \wei{\A(t, x) \xi, \xi}, \quad \|\A\|_{L^\infty} \leq \Lambda,
\end{split}
\end{equation}
for  $\xi \in \bR^{n}$ and for a.e. $(t, x)$.


The most interesting and important feature in the class of equations \eqref{eqn-entire} is that $\mu$ is assumed to be a Muckenhoupt weight function, whose definition will be reviewed shortly. As such, 
the coefficients in \eqref{eqn-entire} could be singular or degenerate in $x_n$-variable. 
When $\mu =1$, \eqref{eqn-entire} and \eqref{main-eqn} are respectively reduced to the linear nondegenerate parabolic equation on $\bR^n$, and on $\bR^n_+$ with the conormal boundary condition.

The singular and degenerate equations \eqref{eqn-entire} and \eqref{main-eqn} are of particular interest due to its geometric and probabilistic applications. We refer to the work \cite{Pop-1, Pop-2, EM} about equations with singular and degenerate coefficients that arise in finance and biology, where regularity theory 
in H\"{o}lder space 
was studied.  
When $\mu(y) = |y|^\alpha$ with $\alpha \in (-1,1)$, the equation
is related the extension problem of the fractional heat operator. See, for instance, \cite{BG17}. In this case, the coefficients can be singular or degenerate on the boundary of the domain. This type of partial differential operators are usually referred to Grushin operators (see \cite{Grushin, Gr-Sa}) and has attracted many attentions in the last few decades; see \cite{Caffa-Sil, Caffa-Stinga, Stinga}. Finally, due to the interest in problems in calculus of variation, the study of regularity of solutions to elliptic and parabolic equations with singular degenerate coefficients already appeared in many classical papers; see  \cite{Fabes-1, Fabes, Str, Stredulinsky, Smith-Stredulinsky, Chia-Sera, Chiarenza-1, Chiarenza-Serapioni}. 

The work mentioned above concerns H\"{o}lder  and Schauder regularity of weak solutions of equations with singular degenerate coefficients.
In this work, we will establish the regularity theory in Sobolev spaces for \eqref{eqn-entire} and \eqref{main-eqn}. In particular, results such as the weighted reverse H\"{o}lder's inequality and Lipschitz estimates are established. Weighted estimates of the Calder\'on-Zygmund type and solvability of solutions in weighted Sobolev spaces are proved, assuming the coefficients are partially \textup{VMO} (vanishing mean oscillations) with respect to the considered weights.

Let us now recall some definitions in order to state our main results. A non-negative locally integrable function $\mu : \bR \rightarrow \bR_+$ is said to be in the $A_2(\bR)$-Muckenhoupt class of weights if it satisfies the following condition:
\begin{equation}  \label{eq11.42}
[\mu]_{A_2(\bR)} :=  \sup_{y \in \bR,\, r>0 }  \left( \fint_{y-r}^{y+r} \mu(s)\,ds \right)  \left( \fint_{y-r}^{y+r} \frac{1}{\mu(s)} \, ds \right)  < \infty.
\end{equation}
For any $r>0$ and $z_0 = (t_0, x'_0, x_{n0})=(z_0',x_{n0}) \in  \bR \times \bR^{n}$,  we denote the partial weighted mean oscillation of the coefficient $\A$ in the cyllinder $Q_r(z_0)$ by
\[
\A^{\#}_{r}(z_0) = \fint_{Q_{r}(z_0)} |\A(t,x) - \bar{\A}_{Q'_{r}(z_0')}(x_n)|\, \mu(dz),
\]
where $\mu(dz) = \mu(x_n) dxdt$ and
$$
\bar{\A}_{Q'_{r}(z_0')}(x_n) = \fint_{Q'_{r}(z_0')}\A(t,x', x_n) \,dx'\, dt.
$$
Moreover, $Q'_r(z_0')$ is the parabolic cylinder in $\bR \times \bR^{n-1}$ centered at $z'_0 =(t_0,x_0') \in \bR\times \bR^{n-1}$ with radius $r$, and $Q_r(z_0)$ is the parabolic cylinder in $\bR \times \bR^{n}$ centered at $z_0 =(t_0, x_0', x_{n0}) \in \bR\times \bR^n$ with radius $r$.  We also denote $Q_r^+(z_0)=Q_r^+(z_0)\cap \{x_n>0\}$. See Section \ref{weighted-Sobolev} for more precise definitions. Also, throughout the paper, the following notation is used
\[
\fint_{Q} f(t, x)\, \omega (dz) = \frac{1}{\omega (Q)} \int_{Q} f(t, x) \, \omega (dz), \quad \text{where} \quad \omega(Q) = \int_{Q} \omega (dz)
\]
for any measurable function $f$ defined on a measurable set $Q \subset \bR^{n+1}$ and for any Borel measure $\omega$ on $\bR^{n+1}$.


Our first result is a local interior estimates in weighted Sobolev spaces for weak solutions of \eqref{eqn-entire}.

\begin{theorem} \label{interior-theorem} Let $q \geq 2$, $K_0 \geq 1$, and $\lambda\ge 0$ be constants. Then there exists $\delta=\delta(n,\Lambda, K_0,q)>0$ such that the following statement holds. Let $z_0=(z_0',x_{n0}) \in \bR^{n+1}$ and $R_0 \in (0,1/4)$. Assume that $\A: Q_2(z_0) \rightarrow \bR^{n\times n}$ such that $\A^{\#}_r (z) \leq \delta$ for any $r \in (0,R_0)$ and $z \in Q_{3/2}(z_0)$. Assume that $\mu$ satisfies the following $A_1$ type condition 
\[
\fint_{y -\rho}^{y+\rho}\mu(s)\, ds \leq K_0 \mu(y), \quad \textup{for a.e.} \ y \in (x_{n0}-2, x_{n0}+2), \quad \forall \ \rho \in (0,1).
\]
Suppose that $u$ is a weak solution of
\eqref{eqn-entire} in $Q_2(z_0)$,
where $F\in L^q(Q_2(z_0),\mu)$, $f=f_1+\sqrt \lambda f_2$, $f_1\in L^{ql_0'}(Q_2(z_0),\mu)$, $f_2\in L^{q}(Q_2(z_0),\mu)$,  $\lambda \geq 0$, and $l_0' = l_0'(n, M_0) \in (0,1)$ is defined below in \eqref{l-0-prime}. Then,
\[
\begin{split}
&\fint_{Q_1(z_0)} (|D u|+\sqrt \lambda|u|)^q \mu(dz)  \leq C \left[ \left( \fint_{Q_2(z_0)} |u|^2 \mu(dz) \right)^{q/2} \right. \\
&\quad \quad + \left . \fint_{Q_2(z_0)} (|\F|+|f_2|)^q \mu(dz) + \left( \fint_{Q_2(z_0)} | f_1|^{ql_0'} \mu(dz) \right)^{1/l_0'} \right],
\end{split}
\]
where $C>0$ is a constant depending on $n$, $\Lambda$, $K_0$, $q$, and $R_0$.
\end{theorem}

We also have the following boundary estimate in weighted Sobolev spaces for weak solutions of \eqref{eqn-entire}-\eqref{main-eqn}.  Note that the condition on $\mu$ is weaker near the boundary.
\begin{theorem}\label{bdr-reg-est}
Let $q \geq 2$, $K_0 \geq 1, M_0 \geq 1, R_0 \in (0, 1/4)$, and $\lambda\ge 0$ be constants. Then there exists $\delta=\delta(n,\Lambda,M_0,K_0,q)>0$ such that the following holds. Assume that $\A :Q_2^+ \rightarrow \bR^{n\times n}$ such that $\A^{\#}_r(z) \leq \delta$ for any $r \in (0, R_0)$ and $z \in Q_{3/2}^+$. Assume that $\mu$ satisfies \eqref{eq11.42} with $[\mu]_{A_2} \leq M_0$ and
\begin{equation}  \label{eq11.42c}
\fint_{y-\rho}^{y+\rho} \mu(s)\,ds\le K_0 \mu(y) \quad \text{for a.e.} \ y \in (0,2) \,\, \text{and any} \  \rho \in (0, y/2).
\end{equation}
Suppose that $u$ is a weak solution of \eqref{eqn-entire}-\eqref{main-eqn} in $Q_2^+$,
where $F\in L^q(Q_2^+,\mu)$, $f=f_1+\sqrt \lambda f_2$, $f_1\in L^{ql_0'}(Q^+_2,\mu)$, $f_2\in L^{q}(Q^+_2,\mu)$, and $l_0' = l_0'(n, M_0) \in (0,1)$ is defined below in \eqref{l-0-prime}.
Then,
\begin{equation} \label{L-q-bdr-est}
\begin{split}
&\fint_{Q_1^+} (|D u|+\sqrt \lambda |u|)^q \mu(dz)   \leq C \left[ \left( \fint_{Q_2^+} |u|^2 \mu(dz)\right)^{q/2} \right. \\
&\quad\quad + \left. \fint_{Q_2^+} (|\F|+|f_2|)^q \mu(dz) + \left( \fint_{Q_2^+} |f_1|^{ql_0'} \mu(dz) \right)^{1/l_0'} \right],
\end{split}
\end{equation}
where $C>0$ is a constant depending on $n$, $\Lambda$, $M_0$, $K_0$, $q$, and $R_0$.
\end{theorem}

From Theorems \ref{interior-theorem} and \ref{bdr-reg-est}, we derive the following solvability results in weighted Sobolev spaces for \eqref{eqn-entire} and \eqref{main-eqn}.
\begin{theorem}  \label{thm1.entire} Let $K_0 \geq 1$, $q \geq 2$ and $R_0 \in (0,1)$ be constants. Then there exist $\delta=\delta(n,\Lambda, K_0,q)>0$ and $\lambda_0=\lambda_0(n,\Lambda, K_0,q,R_0)>0$ such that the following holds. Assume that $\A^{\#}_r(z) \leq\delta$ for any $r \in (0,R_0)$ and $z \in (-\infty,T) \times \bR^n$ with some $T>0$. Assume also that $\mu$ satisfies the following $A_1$-Muckenhoupt type condition
\[
\fint_{y-\rho}^{y+\rho} \mu(s) ds \leq K_0 \mu(y), \quad \textup{for a.e.} \ y \in \bR, \quad \forall \ \rho\ \in (0,1).
\]
Suppose that $\F\in L^q((-\infty,T)\times\bR^n,\mu)$, $f=f_1+\sqrt \lambda f_2$, $f_1\in L^{ql_0'}((-\infty,T)\times\bR^n,\mu)$, $f_2\in L^{q}((-\infty,T)\times\bR^n, \mu)$, and $l_0' = l_0'(n, M_0) \in (0,1)$ is defined below in \eqref{l-0-prime}. Then, for $\lambda\ge\lambda_0$,  there exists a unique weak solution $u$ of \eqref{eqn-entire} in $(-\infty,T)\times \bR^d$, 
which satisfies
\begin{align*}  
&\int_{(-\infty,T)\times\bR^n} (|D u|+\sqrt \lambda |u|)^q \,\mu(dz)  \\
&\leq C\left[ \int_{(-\infty,T)\times\bR^n} (|\F|+|f_2|)^q \,\mu(dz) + \left( \int_{(-\infty,T)\times\bR^n} |f_1|^{ql_0'}\,\mu(dz) \right)^{1/l_0'} \right],
\end{align*}
where $C>0$ depends only on $n$, $\Lambda$,  $K_0$, $q$, and $R_0$.  When $q\in (1,2)$, the result still holds provided that $f_1\equiv 0$.
\end{theorem}

\begin{theorem}  \label{thm1.11} Let $K_0 \geq 1$, $q \geq 2, R_0 \in (0,1)$, and $M_0 \geq 1$ be constants.  Then there exist $\delta=\delta(n, \Lambda, M_0, K_0, q)>0$ and $\lambda_0=\lambda_0(n, \Lambda, M_0, K_0, q, R_0)>0$ such that the following holds.  Assume that $\A : (-\infty,T) \times \bR^{n}_+ \rightarrow \bR^{n\times n}$ such that $\A_r^{\#}(z) \leq \delta$ for any $z \in (-\infty,T) \times \bR^{n}_+$ and $r \in (0, R_0)$.  Assume that $\mu$ satisfies \eqref{eq11.42} with $[\mu]_{A_2} \leq M_0$ and \eqref{eq11.42c}. Suppose $\F\in L^q((-\infty,T)\times\bR^n_+,\mu)$, $f=f_1+\sqrt \lambda f_2$, $f_1\in L^{ql_0'}((-\infty,T)\times\bR^n_+,\mu)$, $f_2\in L^{q}((-\infty,T)\times\bR^n_+,\mu)$, and $l_0' = l_0'(n, M_0) \in (0,1)$ is defined below in \eqref{l-0-prime}.
Then,  for  $\lambda\ge\lambda_0$, there exists a unique weak solution $u$ of \eqref{eqn-entire}-\eqref{main-eqn} in $(-\infty,T)\times \bR^n_+$, 
which satisfies
\begin{equation}  \label{eq5.25}
\begin{split}
&\int_{(-\infty,T)\times\bR^n_+} (|D u|+\sqrt \lambda |u|)^q \,\mu(dz) \\
&\leq C\left[ \int_{(-\infty,T)\times\bR^n_+} (|\F|+|f_2|)^q \,\mu(dz) + \left( \int_{(-\infty,T)\times\bR^n_+} |f_1|^{ql_0'} \,\mu(dz) \right)^{1/l_0'} \right],
\end{split}
\end{equation}
where $C>0$ depends only on $n$, $\Lambda$, $M_0$, $K_0$, $q$, and $R_0$. When $q\in (1,2)$, the result still holds provided that $f_1\equiv 0$.
\end{theorem}

\begin{remark} From the above two theorems, we also obtain the solvability for the corresponding initial value problem in $(0,T)\times \bR^n$ (or $(0,T)\times \bR^n_+$, respectively), where $T\in (0,\infty]$. Indeed, we can solve the equation with $F$ and $f$ being zero for $t\le 0$. By using the unique solvability, it is easily seen that $u=0$ for $t\le 0$. When $T$ is finite, we can take $\lambda_0=0$ if we allow $C$ to also depend on $T$. For example, in the half space case, for $\lambda\in [0,\lambda_0]$, we consider $v=v(t,x)=e^{-\lambda_0 t}u(t,x)$, which satisfies
\begin{equation*} 
\left\{
\begin{array}{cccl}
\mu(x_n) a_0(x_n) v_t - \textup{div}[\mu(x_n) (\mathbf{A}\nabla v - e^{-\lambda_0 t}\F)] +(\lambda+\lambda_0)\mu(x_n) v & = & \mu(x_n)e^{-\lambda_0 t}f, \\
\displaystyle{\lim_{x_n \rightarrow 0^+}\wei{\mu(x_n)(\A\nabla v - e^{-\lambda_0 t}\F), \mathbf{e}_n }} & = &0
\end{array} \right.
\end{equation*}
in $(0,T) \times \bR^n_+$. 
By applying the estimate with $\lambda+\lambda_0$ in place of $\lambda$, and returning to $u$, we get \eqref{eq5.25} the desired estimate with a constant $C$ also depending on $T$.
\end{remark}
Besides  Theorems \ref{interior-theorem}, \ref{bdr-reg-est}, \ref{thm1.entire}, and  \ref{thm1.11},  we emphasize that many other fundamental results for equations with singular degenerate coefficients are also established in this paper such as reverse H\"{o}lder's inequalities (see Section \ref{reverse-H-section}) and Lipschitz regularity estimates (see Section \ref{Lipschits-Section}). These results are non-trivial and completely new even in the elliptic setting. They are also of independent interest and could be useful for other purposes. %

\begin{remark} The following points regarding our main theorems are worth highlighting. 
\begin{itemize}
\item[\textup{(i)}] Condition \eqref{eq11.42c} holds for $\mu(y) = y^\alpha, y >0$ with $\alpha \in (-1,1)$. As a consequence, Theorem \ref{bdr-reg-est} and Theorem \ref{thm1.11} hold true for such $\mu$.
\item[\textup{(ii)}] By localizing the estimates in Theorems \ref{thm1.entire} and \ref{thm1.11}, we can get local interior and boundary estimates similar to these in Theorem \ref{interior-theorem} and \ref{bdr-reg-est} with $q\in (1,2)$.
\item[\textup{(iii)}] If $\mu =1$, Theorems \ref{interior-theorem}, \ref{bdr-reg-est}, \ref{thm1.entire}, and  \ref{thm1.11} hold with $l_0' = \frac{n+2}{n+4}$.
\end{itemize}
\end{remark}
We now provide several comments regarding our main theorems. First, even if $\mu =1$, it is well known from the work \cite{Me} that some type of smallness assumption on the oscillation of the coefficients is necessary. Observe that in the condition $\A^{\#}_r(z) \leq \delta$ in our theorems, the oscillation of the coefficients $\A$ is only measured in the $(t,x')$-variable with respect to the weight $\mu$. This type of condition is usually referred to the partially \textup{VMO} condition first introduced by Kim and Krylov \cite{Kim-Krylov} for the case $\mu =1$. For example, when $\A(t,x) = a(x_n) \tilde{\A}(t, x')$, it can be easily checked that if $\tilde{\A}$ is uniformly elliptic and in \textup{VMO}, then no regularity is required for $a$, except for a uniformly elliptic condition $0 < \gamma \leq a(s) \leq \gamma^{-1}$ with some $\gamma >0$. In this perspective, our results extend the results established in \cite{Kim-Krylov, Kim-Krylov-1, Dong-Kim-1, Dong-Kim}, where the case $\mu =1$ is studied. In this line of research, we also would like to mention the recent work \cite{CMP, MP} in which regularity of weak solutions for elliptic equations were studied. In particular, in \cite{CMP} the coefficients are assumed to be singular, degenerate, and VMO in all the variables. In \cite{MP}, elliptic equations with coefficients that are singular or degenerate in one-variable direction as in this current paper is studied. However, compared to the current work, some quite restrictive additional conditions on $\mu$ and structural conditions on the coefficients $\A$ are required in these two papers. For instance, when the principal coefficient $\A(x) = |x|^\alpha$ with $x \in \bR^n$, results in \cite{CMP} are only valid with $|\alpha|$ sufficiently small.

Our proof is based on the perturbation technique using homogeneous equations with coefficients that are frozen in the $(t,x')$-variables, 
and apply the reverse H\"{o}lder's inequalities and Lipschitz estimates mentioned above. We exploit a level set argument using maximal function free approach introduced in \cite{A-Mingione} and extend it to the weighted setting. In this regard, we note that it is possible to extend our regularity estimate theorems to the estimates in weighted Lorentz spaces. However, we do not pursue this direction to avoid further technical complexities. Let us also mention that the proof of the Lipschitz estimates in Section \ref{Lipschits-Section} is highly delicate because the coefficients are singular and degenerate in $x_n$-variable, and they are only just measurable with no other extra regularity assumptions. To achieve the desired results, we elaborately explore and use the analysis of anisotropic weighted Sobolev inequalities, energy estimates, iteration arguments, and the structure of the considered equations.


The remaining part of the paper is organized as follows. In Section \ref{weighted-Sobolev}, we present several results on weighted Sobolev inequalities. The weighted reverse H\"{o}lder's inequalities for equations with singular degenerate coefficients will be stated and proved in Section \ref{reverse-H-section}. In Section \ref{Lipschits-Section}, we prove Lipschitz estimates for homogeneous equations with singular degenerate coefficients that only depend on one spatial variable. Section \ref{interior-Section} is devoted to the study of interior Sobolev estimates. The proof of Theorem \ref{interior-theorem} is also given in this section. Similarly, Section \ref{boundary-section} contains the proof of Theorem \ref{bdr-reg-est} on boundary estimates in weighted Sobolev spaces. In the last section, Section \ref{last-section}, we prove Theorems \ref{thm1.entire} and \ref{thm1.11}.

\section{Weighted Sobolev inequalities} \label{weighted-Sobolev}

We begin the section with introducing notation used in the paper. For any $\rho >0$ and $x = (x', x_n) \in \bR^{n-1} \times \bR$, we denote the cylinder in $\bR^n$ of radius $\rho$ centered at $x =(x', x_n) \in \bR^{n}$ by
\[
D_\rho(x) = B'_\rho(x') \times (x_n -\rho, x_n + \rho),
\]
where $B'_\rho(x')$ is the ball in $\bR^{n-1}$ with radius $\rho$ and centered at $x' \in \bR^{n-1}$. For any $z = (t, x)=(z',x_n) \in \bR \times \bR^n$, the parabolic cylinder in $\bR \times \bR^n$ of radius $\rho$ centered at $z$ is denoted by $Q_\rho(z) = (t - \rho^2, t] \times D_\rho(x)$.  We also write $\Gamma_\rho(t) = (t - \rho^2, t]$, $Q'_\rho(z') = \Gamma_\rho(t) \times B'_\rho(x')$, and
\[
B'_\rho = B'_\rho (0),\quad D_\rho = D_\rho(0),\quad Q_\rho = Q_\rho(0), \quad Q'_\rho = Q'_\rho(0),\quad \text{and} \quad \Gamma_\rho = \Gamma_\rho(0).
\]
The upper half parabolic cylinders are written as
\[
D^+_\rho(x) = D_\rho(x)\cap \{x_n>0\}, \quad Q_{\rho}^+(z) = \Gamma_\rho(t) \times D^+_\rho(x),
\]
and $D_\rho^+ = D_\rho^+ (0)$ and $Q_\rho^+ =Q_\rho^+(0)$.

Let us also recall that for each $p \in [1, \infty)$, a locally integrable function $\mu : \bR^n \rightarrow \bR^+$ is said to be in $A_p(\bR^n)$ Muckenhoupt class of weights if and only if $[\mu]_{A_p(\bR^n)} < \infty$, where
\begin{equation*}
[\mu]_{A_p(\bR^n)} = \left\{
\begin{array}{lll}
& \displaystyle{\sup_{\rho >0,x \in \bR^n} \left[\fint_{B_\rho(x)} \mu(y)\, dy \right]\left[\fint_{B_\rho(x)} \mu(y)^{\frac{1}{1-p}}\, dy \right]^{p-1}}, & \quad \text{if} \quad p \in (1, \infty) \\
& \displaystyle{\sup_{\rho >0,x \in \bR^n} \left[\fint_{B_\rho(x)} \mu(y)\, dy \right] \norm{\frac{1}{\mu}}_{L^\infty(B_\rho(x))}}, & \quad \text{if} \quad p =1.
\end{array} \right.
\end{equation*}
Observe that if $\mu \in A_p(\bR)$, then $\tilde{\mu} \in A_p(\bR^n)$ with $[\mu]_{A_p(\bR)} = [\tilde{\mu}]_{A_p(\bR^n)}$, where $\tilde{\mu}(x) = \mu(x_n)$ for $x = (x', x_n) \in \bR^n$. Sometimes, if the context is clear, we neglect the spacial domain and only write $\mu \in A_p$.

For any Borel measure $\omega$ on $\bR^n$ and any measurable set $\Omega \subset \bR^n$, the following notation is used
\[
\fint_{\Omega} f(x)\, \omega (dx) = \frac{1}{\omega(\Omega)} \int_{\Omega} f(x) \,\omega (dx)
\]
for any $\omega$-measurable function $f$ defined on $\Omega$.
For a non-empty 
domain $\Omega \subset \bR^n$,  and for a given weight $\mu$ and with some $p \in [1,\infty)$, we denote $L^p(\Omega, \mu)$ to be the weighted Lebesgue space of all measurable functions $f$ defined on $\Omega$ such that
\[
\norm{f}_{L^p(\Omega, \mu)} = \left( \int_\Omega |f(x)|^p\, \mu(dx) \right)^{1/p} < \infty \quad \text{with} \quad \mu(dx) = \mu(x) dx.
\]
Also, $W^{1,p}(\Omega, \mu)$ is the weighted Sobolev spaces consisting of all measurable functions $f$ defined on $\Omega$ such that $f$ and its weak derivatives $D_k f$ are all in $L^p(\Omega, \mu)$ for $k =1,2,\ldots, n$. For any interval $\Gamma\in \bR$, we denote
$$
V^{1,2}(\Gamma\times \Omega,\mu)=
L^2(\Gamma, W^{1,2}(\Omega, \mu)) \cap L^\infty(\Gamma, L^2(\Omega, \mu)).
$$

Throughout the paper, let $\phi_0 \in C_0^\infty(D_2)$ be a fixed standard cut-off function satisfying $\phi_0 \equiv 1$ on $D_1$, $0\le \phi_0 \le 1$, and $|D \phi_0 | \leq 2$. For any $\bar{x} \in \bR^n$ and $\rho >0$, we define $\eta_{\bar{x},\rho}(x) = \phi_0( 2(x-\bar{x})/\rho)$. For each function $u$ defined in the neighborhood of the cylinder $D_{\rho}(\bar{x})$,  following the idea in \cite{Giaquinta-Struwe}, we define the weighted mean of $u$ in $D_\rho(\bar{x})$ with respect to a given weight $\mu : \bR^n \rightarrow (0, \infty)$ as
\begin{equation} \label{weighted-average}
\bar{u}_{\bar{x},\rho} = \fint_{D_\rho(\bar{x})} u(x)   \eta_{\bar{x},\rho}^2(x) \,\mu(dx), \quad \text{where} \quad \mu(dx)=\mu(x)dx.
\end{equation}

The following lemma is a consequence of \cite[Theorem 1.5]{Fabes}.
\begin{lemma}[Weighted Sobolev-Poincar\'e's inequality] \label{Sobolev-Poincare} Let $1 < p < \infty$ and $\mu \in A_p(\bR^n)$ with $[\mu]_{A_p(\bR^n)} \leq M_0$ for some fixed $M_0\geq 1$. Then, there exists $\gamma_0 = \gamma_0 (n, M_0) >0$ such that for any $1 \leq k \leq \frac{n}{n-1} + \gamma_0$
\begin{equation}
                        \label{eq2.21}
\left(\fint_{D_R}| u(x) - A_R|^{kp} \,\mu(dx) \right)^{1/(kp)}  \leq C(n, M_0, k) R \left(\fint_{D_R}  |D u|^p \,\mu(dx) \right)^{1/p}
\end{equation}
for any $u \in W^{1,p}(D_R, \mu)$ and $R \in (0, \infty)$, where
\[
A_R = \bar{u}_{0, R}, \quad \text{or} \quad A_R = \fint_{D_R} u(x) \,dx, \quad \text{or} \quad A_R = \fint_{D_R} u(x) \,\mu(dx).
\]
\end{lemma}
\begin{proof}
For the last two cases when $A_R = \fint_{D_R} u(x) \,dx$ or $A_R = \fint_{D_R} u(x) \,\mu(dx)$, \eqref{eq2.21} was obtained in \cite[Theorem 1.5]{Fabes}. We only treat the case when $A_R =  \bar{u}_{0, R}$. Using H\"older's inequality,
\begin{align*}
|u(x)-A_R|&\le\left(\int_{D_{R}}\eta_{0,R}^2(x)\,\mu(dx) \right)^{-1}
\int_{D_{R}}|u(x)-u(y)|\eta_{0,R}^2(y)\,\mu(dy)\\
&\le\left(\int_{D_{R}}\eta_{0,R}^2(x)\,\mu(dx) \right)^{-1/(kp)}
\left(\int_{D_{R}}|u(x)-u(y)|^{kp}\eta_{0,R}^2(y)\,\mu(dy)\right)^{1/(kp)}.
\end{align*}
By the doubling property of $\mu$, we have
$$
\int_{D_R}\mu(dx)\le C\int_{D_{R/2}}\mu(dx)\le C\int_{D_{R}}\eta_{0,R}^2(x)\,\mu(dx).
$$
Thus the left-hand side of \eqref{eq2.21} is bounded by
\begin{align*}
&\left(\int_{D_{R}}\mu(dx) \right)^{-\frac 1{ kp}}\left(\int_{D_{R}}\eta_{0,R}^2(x)\,\mu(dx) \right)^{-\frac 1{ kp}}
\left(\int_{D_R}\int_{D_R}| u(x) -u(y)|^{kp} \eta_{0,R}^2(y)\,\mu(dx)\,\mu(dy)\right)^{-\frac 1{ kp}}\\
&\le \left(\int_{D_{R}}\,\mu(dx) \right)^{-2/(kp)}
\left(\int_{D_R}\int_{D_R}| u(x) -u(y)|^{kp} \,\mu(dx)\,\mu(dy)\right)^{1/(kp)}\\
&\le \left(\int_{D_{R}}\,\mu(dx) \right)^{-1/(kp)}
\left(\int_{D_R}| u(x) -c|^{kp} \,\mu(dx)\right)^{1/(kp)}
\end{align*}
for any $c\in \bR$. By taking $c=\fint_{D_R} u(x) \,\mu(dx)$, we complete the proof.
\end{proof}

The following weighted parabolic embedding result will be used in our paper.
\begin{lemma}[Weighted parabolic embedding] \label{weighted-Sobolev-imbed} Assume that $\mu : \bR^n \rightarrow (0, \infty)$ such that $\mu \in A_{2}$ with $[\mu]_{A_2} \leq M_0$. Then, there exists a constant $C = C(n,M_0)$ and $l_0 = l_0(n, M_0) >1$ such that
\[
\begin{split}
 \left(\fint_{Q_\rho} |u|^{2l_0} \,{\mu}(dz) \right)^{1/l_0} \leq C \left[ \sup_{t\in \Gamma_\rho} \fint_{D_\rho} |u|^2 \,\mu(dx) +
\rho^2\fint_{Q_\rho} |D u|^2\,{\mu}(dz) \right]
\end{split}
\]
for any $u \in V^{1,2}(Q_\rho,\mu)$, where ${\mu}(dz) = \mu(x) dz$ and $\rho >0$. Moreover, if $\mu =1$, then $l_0 = \frac{n+2}{n}$.
\end{lemma}
\begin{proof} Note that if $\mu =1$, then the inequality is the standard unweighted parabolic embedding theorem with $l_0 = \frac{n+2}{n} >1$. For the general case, 
 let us denote
\[
\hat{u} = u - \tilde{u}_{D_\rho}(t), \quad \text{where} \quad \tilde{u}_{D_\rho}(t) =  \fint_{D_\rho} u(t, x) \,\mu(dx).
\]
Let $\kappa = 2(\frac{n}{n-1} + \gamma_0) >2$, where $\gamma_0$ is defined in Lemma \ref{Sobolev-Poincare}. Also, let $l_0 = 2(1- \frac{1}{\kappa}) >1$. 
Using H\"{o}lder's inequality and Lemma \ref{Sobolev-Poincare}, we obtain
\[
\begin{split}
&  \fint_{D_\rho} |\hat{u}|^{2l_0} \,\mu (dx)
  \leq \left( \fint_{D_\rho} |\hat{u}|^{2} \,\mu(dx) \right)^{1-\frac 2\kappa}  \left( \fint_{D_\rho} |\hat{u}|^{\kappa} \,\mu(dx) \right)^{\frac{2}{\kappa}} \\
 & \leq C  \left(\sup_{t\in \Gamma_\rho} \fint_{D_\rho} |\hat{u}(t, x)|^2 \,\mu(dx)\right)^{1-\frac{2}{\kappa}} \left( \fint_{D_\rho} |D \hat{u}|^{2} \,\mu(dx) \right) \rho^{2}.
\end{split}
\]
Then, by integrating in time, and also using Young's inequality, we obtain
\[
\left( \fint_{Q_\rho} |\hat{u}|^{2l_0} \,{\mu} (dz)\right)^{\frac{1}{2l_0}} \leq
C\left[\left( \sup_{t\in \Gamma_\rho} \fint_{D_\rho} |\hat{u}(t, x)|^2 \,\mu(dx) \right)^{1/2}+ \rho\left(  \fint_{Q_\rho} |D \hat{u}|^{2} \,{\mu}(dz)\right)^{1/2}  \right].
\]
This estimate together with the triangle inequality imply that
\[
\begin{split}
& \left(\fint_{Q_\rho} |u|^{2l_0} \,{\mu}(dz)  \right)^{\frac{1}{2l_0}} \leq \left(\fint_{Q_\rho} |\hat{u}|^{2l_0} \,{\mu}(dz)  \right)^{\frac{1}{2l_0}} + \left(\fint_{Q_\rho} |\tilde{u}_{D_\rho}(t)|^{2l_0} \,{\mu}(dz)  \right)^{\frac{1}{2l_0}} \\
& \leq C \left[ \left(\sup_{t\in \Gamma_\rho}  \fint_{D_\rho} |\hat{u}(t, x)|^2 \,\mu(dx) \right)^{1/2} + \rho \left( \fint_{Q_\rho} |D u|^2\,{\mu}(dz) \right)^{1/2} \right] + \sup_{t\in \Gamma_\rho}  |\tilde{u}_{D_\rho}(t)|.
\end{split}
\]
On the other hand, from the H\"{o}lder's inequality, it follows that
\[
 \sup_{t\in \Gamma_\rho} |\tilde{u}_{D_\rho}(t)| \leq  \left(\sup_{t\in \Gamma_\rho}  \fint_{D_\rho} |u(t, x)|^2 \,\mu(dx)\right)^{1/2}.
\]
Then, the estimate in the lemma follows from the last two estimates.
\end{proof}

\section{Weighted reverse H\"{o}lder's inequalities} \label{reverse-H-section}

In this section, we prove self-improving integrability property for the gradients of weak solutions of parabolic equations with singular-degenerate coefficients as in \eqref{eqn-entire}-\eqref{main-eqn}. These types of estimates are usually referred to Meyers-Gehring's estimates, which are needed later in our paper. 
We establish the results for general $A_2$ weight on $\bR^n$ or $\bR^n_+$. In this section, the leading coefficients $\A$ are {\em not} required to have small oscillation. Throughout this paper, we denote $l_0'  = \frac{l_0}{ 2l_0 -1} \in (0,1)$, where $l_0>1$ is defined in Lemma \ref{weighted-Sobolev-imbed}, i.e.,
\begin{equation} \label{l-0-prime}
1/l_0 + 1/l_0' = 2.
\end{equation}
\subsection{Interior weighted reverse H\"{o}lder's inequality}
Throughout this subsection, we assume that $\A : Q_2 \rightarrow \bR^{n\times n}$ is measurable and uniformly elliptic.
Assume also that $\mu : \bR^n \rightarrow (0, \infty)$ is a weight in $A_2(\bR^n)$, and denote $\mu(dx)=\mu(x)dx$ and $ \mu(dz)=\mu(x)dz$. Assume that $\F \in L^2(Q_2, \mu))$ and $f \in  L^{2l_0'}(Q_2, \mu)$.
We consider the equation
\begin{equation} \label{Q-2-d.eqn}
\mu(x)  a_0(x_n) u_t - \textup{div}[\mu(x)( \mathbf{A}  \nabla u - \F)]
= \mu(x)f \quad\text{in}\,\,   Q_2.
\end{equation}
By weak solution of \eqref{Q-2-d.eqn} we mean a function $u \in V^{1,2}(Q_2,\mu)$ satisfying
\begin{equation} \label{weak-formula}
-\int_{Q_2} \mu(x) a_0(x_n) u \varphi_t \,dz + \int_{Q_2} \mu(x)\wei{ (\mathbf{A}  \nabla u - \F), \nabla \varphi}
\,dz= \int_{Q_2} \mu(x)f \varphi\,dz
\end{equation}
for any $ \varphi \in C_0^\infty(Q_2)$.
It turns out that weak solutions of \eqref{Q-2-d.eqn} posses some modest regularity in time variable. To work with test functions involving the solutions, it is convenient to recall the Steklov averages: for each function $g \in Q_r (\bar{z})$, where $\bar{z} = (\bar{t}, \bar{x}) \in \bR^{n+1}$, we denote
\[
[g]_h(t,x) = \left\{
\begin{array}{cl}
\frac{1}{h} \int_{t}^{t +h} g(s,x)\, ds & \quad \text{for} \quad t \in ( \bar{t} -r^2, \bar{t} - h), \\
0 & \quad \text{for} \quad t > \bar{t} -h.
\end{array} \right.
\]
Then, as in \cite[p. 17-18]{DiB}, we can prove that \eqref{weak-formula} is equivalent to
\begin{align} \label{Steklov-weak-formula}
&\int_{ D_2}  \mu(x) a_0(x_n) \partial_t  [u]_h  \phi(x) \,dx + \int_{ D_2} \mu(x) \wei{[\A \nabla u]_h , \nabla \phi}
\,dx  \nonumber\\
&=  \int_{ D_2} \mu(x)\Big[ \wei{  [\F]_h, \nabla \phi} + [f]_{h} \phi (x) \Big] \,dx,
 \end{align}
for any $\phi \in W^{1,2}_0(D_2, \mu)$ and for a.e. $t \in (-4, 4 -h)$. 


\begin{lemma}[Weighted Caccioppoli's type estimate]
                \label{inter-energy}  Let $a_0\equiv 1$. There exists a constant $C = C(n, \Lambda, M_0)$ such that for any $\bar{z} = (\bar{t}, \bar{x}) \in Q_{3/2},\ \rho \in (0,1/4)$, and for any weak solution $u$ of \eqref{Q-2-d.eqn} with $[\mu]_{A_2} \leq M_0$, it holds that
\[
\begin{split}
&  \sup_{\tau \in \Gamma_{\rho}(\bar{t})} \rho^{-2}\fint_{D_{\rho}(\bar{x})} \hat{u} ^2(\tau ,x)   \,\mu(dx) +   \fint_{Q_{\rho}(\bar{z})}  |D \hat{u}|^2  \,{\mu}(dz) \\
& \leq C\left[ \rho^{-2} \fint_{Q_{2\rho}(\bar{z})} \hat{u}^2 \,{\mu}( dz)   + \fint_{Q_{2\rho}(\bar{z})}|\F|^2 \,{\mu}(dz) + \rho^2 \left( \fint_{Q_{2\rho}(\bar{z})} |f |^{2l_0'}  \,{\mu}(dz) \right)^{1/l_0'}\right],
\end{split}
\]
where $\hat{u} = u(t,x) - \bar{u}_{\bar{x}, 2\rho}(t)$ with $\bar{u}_{\bar{x}, 2\rho}(t)$ being defined as in \eqref{weighted-average}.
\end{lemma}
\begin{proof}
Let $\sigma \in C_0^\infty(\bar{t}-4\rho^2,\bar{t}+4\rho^2)$ be a smooth function satisfying $\sigma =1$ on $\Gamma_\rho(\bar{t})$ and $|\sigma'(t)| \leq 4\rho^{-2}$. By taking $\phi = \hat{u}(t,x)  \eta_{\bar{x}, 2\rho}^2(x)  \sigma(t)$ as the test function in  \eqref{Steklov-weak-formula}, integrating in $t$, 
and taking $h \rightarrow 0^+$, we obtain
\[
\begin{split}
& \frac{1}{2}\int_{\bar{t} - 4\rho^2}^\tau \sigma(t) \frac{d}{dt} \left[ \int_{D_{2\rho}(\bar{x})} \hat{u} ^2 \eta_{\bar{x}, 2\rho}^2(x)  \,\mu(dx) \right] \,dt + \int_{\bar{t} - 4\rho^2}^\tau  \int_{D_{2\rho}(\bar{x})}  \partial_t \bar{u}_{\bar{x}, 2\rho}(t) \sigma(t) \eta_{\bar{x}, 2\rho}^2 (t) \hat{u}\,  \mu(dz ) \\
& = - \int_{\bar{t} -4\rho^2}^\tau \int_{D_{2\rho}(\bar{x})} \wei{\A \nabla u - \F, \nabla [\hat{u} \eta_{\bar{x}, 2\rho}^2(x) ] }
 \sigma(t) \,\mu(dz)+ \int_{\bar{t} -4\rho^2}^\tau \int_{D_{2\rho}(\bar{x})} f\hat{u} \eta_{\bar{x}, 2\rho}^2(x)  \sigma(t) \,\mu(dz)
\end{split}
\]
for a.e. $\tau \in (\bar{t} - 4\rho^2, \bar{t})$. Observe that by the definition of $\hat{u}$ it follows that
\begin{equation*}
\int_{\bar{t} - 4\rho^2}^\tau  \int_{D_{2\rho}(\bar{x})}  \partial_t \bar{u}_{\bar{x}, 2\rho}(t) \sigma(t) \eta_{\bar{x}, 2\rho}^2 (t)\, \hat{u}(t,x)  \,\mu(dz)=0.
\end{equation*}
Hence,
\[
\begin{split}
& \frac{1}{2} \int_{D_{2\rho}(\bar{x})} \hat{u} ^2(\tau ,x) \eta_{\bar{x}, 2\rho}^2(x) \sigma(\tau)  \,\mu(dx ) +  \int_{\bar{t} -4\rho^2}^\tau \int_{D_{2\rho}(\bar{x})} \wei{ \A \nabla \hat{u}, \nabla \hat{u} } \eta_{\bar{x}, 2\rho}^2(x)  \sigma(t)\,\mu(dz)  \\
& = - 2 \int_{\bar{t} -4\rho^2}^\tau \int_{D_{2\rho}(\bar{x})} \wei{ \A \nabla \hat{u}, \nabla  \eta_{\bar{x}, 2\rho}(x)}  \hat{u} \eta_{\bar{x}, 2\rho}(x)    \sigma(t) \,\mu(dz)  \\
& \quad + \int_{\bar{t} -4\rho^2}^\tau \int_{D_{2\rho}(\bar{x})} \wei{\F,  \eta_{\bar{x}, 2\rho}^2 \nabla \hat{u} + 2 \hat{u} \eta_{\bar{x},2\rho}  \nabla \eta_{\bar{x},2\rho}  }  \sigma(t) \,\mu(dz) \\
& \quad + \frac{1}{2}\int_{\bar{t} -4\rho^2}^\tau  \int_{D_{2\rho}(\bar{x})} \hat{u} ^2(\tau ,x) \eta_{\bar{x}, 2\rho}^2(x) \sigma'(t)  \, \mu( dz)   + \int_{\bar{t} -4\rho^2}^\tau \int_{D_{2\rho}(\bar{x})} f\hat{u} \eta_{\bar{x}, 2\rho}^2(x)  \sigma(t)\,\mu(dz).
\end{split}
\]
From this and \eqref{ellipticity}, it follows that
\begin{equation} \label{cacao-1}
\begin{split}
& \frac{1}{8 \rho^2} \sup_{\tau}\fint_{D_{2\rho}(\bar{x})} \hat{u} ^2(\tau ,x) \eta_{\bar{x}, 2\rho}^2(x) \sigma(\tau)\,\mu(dx)  +    \Lambda^{-1}\fint_{Q_{2\rho}(\bar{z})} |D \hat{u}|^2 \eta_{\bar{x}, 2\rho}^2(x)  \sigma(t)  \,{\mu}(dz)  \\
& \leq C(\Lambda) \fint_{Q_{2\rho}(\bar{z})} |D \hat{u}| |D \eta_{\bar{x}, 2\rho}|  |\hat{u} |\eta_{\bar{x}, 2\rho} \sigma(t) \,{\mu}(dz)  \\
&\quad + \fint_{Q_{2\rho}(\bar{z})}  |\F| \Big[ \eta_{\bar{x}, 2\rho}^2 |D \hat{u}| + 2 |\hat{u}| \eta_{\bar{x},2\rho} | D \eta_{\bar{x},2\rho}|   \Big]  \sigma(t) \,{\mu}(dz)  \\
& \quad + \frac{1}{2} \fint_{Q_{2\rho}(\bar{z})}  \hat{u} ^2(\tau ,x) \eta_{\bar{x}, 2\rho}^2(x) |\sigma'(t)| \,{\mu}(dz)   +  \fint_{Q_{2\rho}(\bar{z})}|f| |\hat{u}| \eta_{\bar{x}, 2\rho}^2(x)  \sigma(t) \,{\mu}(dz).
\end{split}
\end{equation}
Then, by using H\"{o}lder's inequality and Young's inequality to the first two terms on the right-hand side of \eqref{cacao-1}, we obtain
\[
\begin{split}
& \frac{1}{8 \rho^2} \sup_{\tau}\fint_{D_{2\rho}(\bar{x})} \hat{u} ^2(\tau ,x) \eta_{\bar{x}, 2\rho}^2(x) \sigma(\tau)\,\mu(dx) +   \Lambda^{-1} \fint_{Q_{2\rho}(\bar{z})} |D [ \hat{u} \eta_{\bar{x}, 2\rho}(x)]|^2
\sigma(t) \,{\mu}(dz)  \\
& \leq \frac{1}{2\Lambda}  \fint_{Q_{2\rho}(\bar{z})} |D [\hat{u} \eta_{\bar{x}, 2\rho}(x)]|^2  \sigma(t) \,{\mu}(dz)  \\
& \quad \quad + C(\Lambda) \left[ \rho^{-2} \fint_{Q_{2\rho}(\bar{z})} \hat{u}^2 \,{\mu}(dz)  + \fint_{Q_{2\rho}(\bar{z})} |\F|^2  \,{\mu}(dz) +   \fint_{Q_{2\rho}(\bar{z})}|f| |\hat{u}| \eta_{\bar{x}, 2\rho}^2(x)  \sigma(t) \,{\mu}(dz) \right].
\end{split}
\]
From this, it follows that
\begin{equation} \label{cacao-2}
\begin{split}
& \frac{1}{ \rho^2} \sup_{\tau}\fint_{D_{2\rho}(\bar{x})} \hat{u} ^2(\tau ,x) \eta_{\bar{x}, 2\rho}^2(x) \sigma(\tau)\,\mu(dx) +   \fint_{Q_{2\rho}(\bar{z})} |D [ \hat{u} \eta_{\bar{x}, 2\rho}(x)]|^2  \sigma(t) \,{\mu}(dz)  \\
& \leq C(\Lambda)\left[ \rho^{-2} \fint_{Q_{2\rho}(\bar{z})} \hat{u}^2 \,{\mu}(dz)  + \fint_{Q_{2\rho}(\bar{z})} |\F|^2  \,{\mu}(dz) +   \fint_{Q_{2\rho}(\bar{z})}|f| |\hat{u}| \eta_{\bar{x}, 2\rho}^2(x)  \sigma(t) \,{\mu}(dz) \right].
\end{split}
\end{equation}
We now control the last term on the right-hand side of \eqref{cacao-2}. We write $v(t,x)= \hat{u}(t,x) \eta_{\bar{x}, 2\rho}(x) \sigma(t)$. Note that by \eqref{l-0-prime}, $1 = \frac{1}{2l_0} + \frac{1}{2l_0'}$. Then, by using the H\"{o}lder's inequality, Young's inequality, and then Lemma \ref{weighted-Sobolev-imbed}, we obtain
\[
\begin{split}
& \fint_{Q_{2\rho}(\bar{z})}|f| |\hat{u}| \eta_{\bar{x}, 2\rho}^2(x)  \sigma(t) \,{\mu}(dz) \\
& \leq \left( \fint_{Q_{2\rho}(\bar{z})}\Big| |f| \eta_{\bar{x}, 2\rho}(x) \Big|^{2l_0'}   \,{\mu}(dz) \right)^{1/(2l_0')} \left( \fint_{Q_{2\rho}(\bar{z})} |v|^{2l_0} \,{\mu}(dz) \right )^{1/(2l_0)}  \\
& \leq \frac{\epsilon}{\rho^2} \left( \fint_{Q_{2\rho}(\bar{z})} |v|^{2l_0} \,{\mu}(dz) \right )^{1/l_0} + C(\epsilon) \rho^2  \left( \fint_{Q_{2\rho}(\bar{z})}\Big| |f| \eta_{\bar{x}, 2\rho} (x) \Big|^{2l_0'}   \,{\mu}(dz) \right)^{1/l_0'} \\
& \leq C(n, M_0) \epsilon \left[ \rho^{-2}\sup_{\tau} \fint_{D_{2\rho}} |v(\tau, x)|^2 \,\mu(dx) + \fint_{Q_{2\rho}} |D v(t, x)|^2 \,{\mu}(dz)  \right] \\
&\quad + C(\epsilon) \rho^2  \left( \fint_{Q_{2\rho}(\bar{z})} |f|^{2l_0'}   \,{\mu}(dz) \right)^{1/l_0'},
\end{split}
\]
for any $\epsilon >0$.  From this estimate, \eqref{cacao-2} and by choosing $\epsilon$ sufficiently small, we infer that
\[
\begin{split}
&  \sup_{\tau}\rho^{-2}\fint_{D_{2\rho}(\bar{x})} \hat{u} ^2(\tau ,x) \eta_{\bar{x}, 2\rho}^2(x) \sigma(\tau)\,\mu(dx) +  \fint_{Q_{2\rho}(\bar{z})} |D [ \hat{u}(t,x) \eta_{\bar{x}, 2\rho}(x)]|^2  \sigma(t) \,{\mu}(dz)  \\
& \leq C \left[ \rho^{-2} \fint_{Q_{2\rho}(\bar{z})} \hat{u}^2 \,{\mu}(dz)  + \fint_{Q_{2\rho}(\bar{z})} |\F|^2  \,{\mu}(dz) + \rho^2  \left( \fint_{Q_{2\rho}(\bar{z})} |f|^{2l_0'}   \,{\mu}(dz) \right)^{1/l_0'}\right].
\end{split}
\]
This estimate and the doubling property of $\mu$ implies the desired estimate in the lemma. The proof of the lemma is then complete.
\end{proof}

\begin{lemma} \label{pre-R-Holder}  Let $a_0\equiv 1$. Let $\mu: \bR^n \rightarrow \bR$ such that $\mu \in A_2(\bR^n)$ with $[\mu]_{A_2} \leq M_0$ for some fixed $M_0 >1$. There is $\kappa_0 \in (1,2)$ which depends only on $n, M_0$ such that the following holds. For any $\epsilon >0$, there exists $C_0 = C_0(n,\Lambda, M_0, \epsilon)$ such that for any $\bar{z} = (\bar{t}, \bar{x}) \in Q_{3/2},\ \rho \in (0,1/8)$, and for any weak solution $u$ of \eqref{Q-2-d.eqn}, it holds that
\begin{equation}
                   \label{eq10.55}
\begin{split}
&  \fint_{Q_{\rho}(\bar{z})} |Du|^2 \,{\mu}(dz)   \leq \epsilon \fint_{Q_{4\rho}(\bar{z})} |Du|^2 \,{\mu}(dz)  \\
 & \quad +  C_0\left[\left( \fint_{Q_{4\rho}(\bar{z})} |Du|^{\kappa_0} \,{\mu}(dz) \right)^{2/\kappa_0} +  \fint_{Q_{4\rho}(\bar{z})} |\F|^2  \,{\mu}(dz) +  \rho^2 \left( \fint_{Q_{4\rho}(\bar{z})} |f|^{2l_0'} \,{\mu}(dz) \right)^{1/l_0'}\right].
\end{split}
\end{equation}
\end{lemma}
\begin{proof}
Denote
\[
\mathcal{F}(\rho) =  \fint_{Q_{\rho}(\bar{z})} |\F|^2 \,{\mu}(dz) +  \rho^2 \left( \fint_{Q_{\rho}(\bar{z})} |f|^{2l_0'} \,{\mu}(dz) \right)^{1/l_0'}.
\]
By the Poincar\'e inequality, Lemma \ref{Sobolev-Poincare}, it follows that
\[
\rho^{-2} \fint_{Q_{2\rho}(\bar{z})} |u - \bar{u}_{\bar{x}, 2\rho}(t)|^2 \,{\mu}(dz) \leq C(n, M_0) \fint_{Q_{2\rho}(\bar{z})} |D u|^2 \,{\mu}(dz).
\]
This estimate and Lemma \ref{inter-energy} imply that
\begin{equation} \label{R-2rho}
\rho^{-2} \sup_{t \in \Gamma_{\rho}(\bar{t})} \fint_{D_{\rho}(\bar{x})} |u - \bar{u}_{\bar{x}, 2\rho}(t)|^2 \,\mu(dx) \leq C(n,\Lambda, M_0) \left[  \fint_{Q_{2\rho}(\bar{z})} |D u|^2 \,{\mu}(dz) + \mathcal{F}(2\rho) \right].
\end{equation}
Then, with the notation that $\hat{u} = u - u_{\bar{x}, 2\rho}(t)$, it follows from H\"{o}lder's inequality that
\begin{equation} \label{R-4rho}
\begin{split}
&\rho^{-2}\fint_{Q_{2\rho}(\bar{z})} |\hat{u}|^2 \,{\mu}(dz) \\
&\leq \rho^{-2}
\left[ \sup_{t\in \Gamma_{2\rho}(\bar{t})} \left( \fint_{D_{2\rho}(\bar{x})} |\hat{u}|^2 \,\mu(dx) \right)^{1/2}\right] \left[ \fint_{\Gamma_{2\rho}(\bar{t})} \left( \fint_{D_{2\rho}(x_0)} |\hat{u}|^2 \,\mu(dx) \right)^{1/2} \,dt \right] \\
& \leq C \rho^{-1} \left[\left( \fint_{Q_{4\rho}(\bar{z})} |D u|^2 \,{\mu}(dz) \right)^{1/2}+ \mathcal{F}(4\rho)^{1/2} \right]  \left[ \fint_{\Gamma_{2\rho}(\bar{t})} \left( \fint_{D_{2\rho}(\bar{x})} |\hat{u}|^2 \,\mu(dx) \right)^{1/2} \,dt \right],
\end{split}
\end{equation}
where we have used \eqref{R-2rho} in the last estimate with $\rho$ replaced by $2\rho$. We now control the last factor on the right-hand side of \eqref{R-4rho}. Let us denote $\kappa= 2 \Big ( \frac{n}{n-1} + \gamma_0 \Big) >2$, where $\gamma_0$ is a number depending on $n$ and $M_0$ defined in Lemma \ref{Sobolev-Poincare}. Then, let $\kappa_0 \in (1,2)$ such that $\frac{1}{\kappa} + \frac{1}{\kappa_0} =1$. From this, H\"{o}lder's inequality and the Sobolev-Poincar\'e inequality (see Lemma \ref{Sobolev-Poincare}), it follows that
\[
\begin{split}
&\rho^{-1} \fint_{\Gamma_{2\rho}(\bar{t})} \left( \fint_{D_{2\rho}(\bar{x})} |\hat{u}|^2 \,\mu(dx) \right)^{1/2} \,dt  \\
& \leq \rho^{-1} \fint_{\Gamma_{2\rho}(\bar{t})} \left[ \left( \fint_{D_{2\rho}(\bar{x})} |\hat{u}|^{\kappa_0} \,\mu(dx) \right)^{\frac{1}{2\kappa_0}}  \left( \fint_{D_{2\rho}(\bar{x})} |\hat{u}|^{\kappa} \,\mu(dx) \right)^{\frac{1}{2 \kappa}}  \right] \,dt \\
& \leq C(n, M_0)  \fint_{\Gamma_{2\rho}(t_0)} \left[ \left( \fint_{D_{2\rho}(\bar{x})} |D u|^{\kappa_0} \,\mu(dx) \right)^{\frac{1}{2\kappa_0}}  \left( \fint_{D_{2\rho}(\bar{x})} |D u|^{2} \,\mu(dx) \right)^{\frac{1}{4} } \right] \,dt.
\end{split}
\]
We then use H\"{o}lder's inequality 
for the time integration in the last estimate to infer that
\begin{align*}
&\rho^{-1} \fint_{\Gamma_{2\rho}(\bar{t})} \left( \fint_{D_{2\rho}(\bar{x})} |\hat{u}|^2 \,\mu(dx) \right)^{1/2} \,dt\\
&\leq C(n,M_0)  \left( \fint_{Q_{2\rho}(\bar{z})} |D u|^{\kappa_0} \,{\mu}(dz) \right)^{\frac{1}{2\kappa_0}} \left ( \fint_{Q_{2\rho}(\bar{z})} |D u|^{2} \,{\mu}(dz) \right)^{\frac{1}{4}}.
\end{align*}
The last estimate, together with \eqref{R-4rho} and Young's inequality, implies that
\[
\begin{split}
&\rho^{-2}\fint_{Q_{2\rho}(\bar{z})} |\hat{u}|^2 \,{\mu}(dz)\leq C(n, M_0)\left[\left( \fint_{Q_{4\rho}(\bar{z})} |D u|^2 \,{\mu}(dz) \right)^{1/2}+ \mathcal{F}(4\rho)^{1/2} \right]\\
&\quad\cdot \left( \fint_{Q_{4\rho}(\bar{z})} |D u|^{\kappa_0} \,{\mu}(dz) \right)^{\frac{1}{2\kappa_0}} \left ( \fint_{Q_{4\rho}(\bar{z})} |D u|^{2} \,{\mu}(dz) \right)^{\frac{1}{4}} \\
& \leq \epsilon \fint_{Q_{4\rho}(\bar{z})} |D u|^2 \,{\mu}(dz)  + C(n, M_0, \epsilon) \left[\left( \fint_{Q_{4\rho}(\bar{z})} |D u|^{\kappa_0} \,{\mu}(dz) \right)^{\frac{2}{\kappa_0}}  + \mathcal{F}(4\rho) \right].
\end{split}
\]
From this estimate and Lemma \ref{inter-energy}, we get \eqref{eq10.55}
and the lemma is proved.
\end{proof}

The following result is the main result of this section.
\begin{proposition}[Weighted reverse H\"{o}lder's inequality] \label{reverse-holder-inter}
For each $\Lambda>0, M_0 \geq 1$, there exists $\epsilon_0 \in (0,1)$  depending only on $n$, $\Lambda$, and $M_0$ such that the following statement holds. If $\mu \in A_2(\bR^n)$ with $[\mu]_{A_2} \leq M_0$, we assume $(1+n)(1/l_0'-1)\le 1$, and if $\mu \in A_2(\bR)$ with $[\mu]_{A_2} \leq M_0$, we assume $(3+n)(1/l_0'-1)\le 2$. 
Then for any $q \in [2, 2+ \epsilon_0]$, there exists $C_0 = C_0(n,\Lambda,M_0, q)$ such for any weak solution $u$ of \eqref{Q-2-d.eqn}, it holds that
\begin{equation}
                            \label{eq11.16}
\begin{split}
&  \fint_{Q_{r}(\bar{z})} |Du|^q \,{\mu}(dz)   \\
&\leq  C_0\left[\left( \fint_{Q_{2r}(\bar{z})} |Du|^{2} \,{\mu}(dz) \right)^{q/2} +  \fint_{Q_{2r}(\bar{z})} |\F|^q  \,{\mu}(dz) + \mathcal{G}_{Q_2} (f)^q  \fint_{Q_{2r}(\bar{z})} |f|^{l_0'q} \,{\mu}(dz) \right]
\end{split}
\end{equation}
for any $\bar{z} \in Q_{3/2}$ and $r \in (0,1/2)$, where
\begin{equation}
                \label{G-function}
 \mathcal{G}_{Q_2} (f) =  \left( \fint_{Q_{2}} |f|^{2l_0'} \,{\mu}(dz) \right)^{\frac{1}{2l_0'}- \frac{1}{2}}.
\end{equation}
\end{proposition}
\begin{proof} 
If $\mu \in A_2(\bR^n)$ with $[\mu]_{A_2(\bR^n)} \leq M_0$, it follows that
\[
\frac{{\mu}(Q_2)}{{\mu}(Q_{r}(\bar{z}) )} = \frac{4 \mu(D_2)}{r^2 \mu(D_{r}(\bar x)}  \leq \frac{C(n,M_0)}{r^2} \left(\frac{|D_2|}{|D_r(\bar x)|}\right)^2 \leq C(n, M_0) r^{-2(n+1)}.
\]
Similarly, if $\mu \in A_2(\bR)$ with $[\mu]_{A_2(\bR)} \leq M_0$, it also holds that
\[
\frac{{\mu}(Q_2)}{{\mu}(Q_{r}(\bar{z})} \leq \frac{C(M_0, n)}{r^{n+1}} \left(\frac{|(-2,2)|}{|(\bar x-r,\bar x+r)|}\right)^2 \leq C(M_0, n) r^{-(n+3)}.
\]
Then, by the condition on $l_0'$, we obtain
\[
r^{2} \left( \frac{{\mu}(Q_2)}{{\mu}(Q_{r}(\bar{z})}\right)^{\frac{1}{l_0'} -1} \leq C(M_0, n), \quad \forall r \in (0,1/2), \quad \forall \ \bar{z} \in Q_{3/2}.
\]
Therefore, for any $\bar{z} \in Q_{3/2}$ and $\rho \in (1, 1/8)$,
\begin{align*}
&\rho^2 \left( \fint_{Q_{4\rho}(\bar{z})} |f|^{2l_0'} \,{\mu}(dz) \right)^{1/l_0'}=\rho^2 \left( \fint_{Q_{4\rho}(\bar{z})} |f|^{2l_0'} \,{\mu}(dz) \right)^{1/l_0'-1}\left( \fint_{Q_{4\rho}(\bar{z})} |f|^{2l_0'} \,{\mu}(dz) \right)\\
&\le  \rho^2 \left(\frac{{\mu}(Q_2)}{ {\mu}(Q_{4\rho}(z_0)} \right)^{\frac{1}{l_0'} -1}\left( \fint_{Q_{2}} |f|^{2l_0'} \,{\mu}(dz) \right)^{1/l_0'-1}\left( \fint_{Q_{4\rho}(\bar{z})} |f|^{2l_0'} \,{\mu}(dz) \right) \\
&\le C(M_0, n) \mathcal{G}^2_{Q_2}(f)\left( \fint_{Q_{4\rho}(\bar{z})} |f|^{2l_0'} \,{\mu}(dz) \right) = C(M_0, n)  \fint_{Q_{4\rho}(\bar{z})} |\tilde{f}|^{2} \,{\mu}(dz),
\end{align*}
where $\tilde{f} =  \mathcal{G}_{Q_2}(f) |f|^{l_0'}$. Then in the special case when $a_0\equiv 1$, from the above estimate and Lemma \ref{pre-R-Holder}, there is $\kappa_0 \in (1,2)$ such that  with $\epsilon \in (0,1)$
\begin{equation*}
\begin{split}
&  \fint_{Q_{\rho}(\bar{z})} |Du|^2 \,{\mu}(dz)   \leq \epsilon \fint_{Q_{4\rho}(\bar{z})} |Du|^2 \,{\mu}(dz)  \\
 & \quad \quad +  C_0(\Lambda, M_0, n, \epsilon)\left[\left( \fint_{Q_{4\rho}(\bar{z})} |Du|^{\kappa_0} \,{\mu}(dz) \right)^{2/\kappa_0} +  \fint_{Q_{4\rho}(\bar{z})} |\F|^2  \,{\mu}(dz) +  \fint_{Q_{4\rho}(\bar{z})} |\tilde{f}|^{2} \,{\mu}(dz) \right]
\end{split}
\end{equation*}
for any $\rho \in (0, 1/8)$ and $\bar{z} \in Q_{3/2}$. Now  \eqref{eq11.16} is a consequence of the well-known Gehring lemma; see, for instance, \cite[Ch. V]{Giaquinta}.

 For the general case, we can first absorb $a_0$ to $\mu$ and replace $\A, \F$ and $f$ with $\A/a_0, \F/a_0$ and $f/a_0$, respectively. The proposition is proved.
\end{proof}

\begin{remark} In the special case when $\mu =1$, Proposition \ref{reverse-holder-inter} as well as Proposition \ref{reverse-holder-bdr} below still hold with $l_0=(n+2)/n$ and $l_0'=(n+2)/(n+4)$. In this case, we have
\[
\frac{{\mu}(Q_2)}{{\mu}(Q_{r}(\bar{z}))} \leq C(n) r^{-(n+2)} \quad \text{and}\quad (2+n)(1/l_0'-1)= 2.
\]
Therefore, the proofs in the subsequent sections also work when $\mu =1$ and $l_0=(n+2)/n$.
\end{remark}

\subsection{Weighted reverse H\"{o}lder's inequality on flat boundary}
Throughout this subsection, we assume that $\A = (\A_{ij})_{i,j =1,\ldots, n}: Q_2^+: \rightarrow \bR^{n\times n}$ is measurable and uniformly elliptic. 
Assume also that $\mu $ is an $A_2$ weight either on $\bR^n$ or $\bR$. For each measurable vector field $\F: Q_2^+ \rightarrow \bR^n$ and each measurable function $f :Q_\rho^+ \rightarrow \bR$ satisfying  $|\F|\in L^2(Q_2^+, {\mu}), |f| \in L^{2l_0'}(Q_2^+, {\mu})$, we consider the equation
\begin{equation} \label{Q-2-bdr.eqn}
\left\{
\begin{array}{ccll}
\mu(x)  a_0(x_n) u_t - \textup{div}[\mu(x) (\mathbf{A}  \nabla u - \F)]  & =  & \mu(x)f \\
\displaystyle{\lim_{x_n \rightarrow 0^+}\wei{\mu(x)(\A \nabla u - \F), \mathbf{e}_n }} & = &0
\end{array} \right. \text{in}\,\,   Q_2^+.
\end{equation}
By weak solution of \eqref{Q-2-bdr.eqn} we mean a function $u \in V^{1,2}(Q_2^+,\mu)$ satisfying
\begin{align*}  
-\int_{Q_2^+}\mu(x) a_0(x_n) u\varphi_t \,dz + \int_{Q_2^+} \mu(x )\wei{ (\mathbf{A} \nabla u - \F), \nabla \varphi} \,dz = \int_{Q_2^+} \mu(x)f \varphi \,dz
\end{align*}
for any $\varphi \in C_0^\infty(Q_2)$.
The following weighted reverse H\"{o}lder's inequality can be proved exactly the same way as that of Proposition \ref{reverse-holder-inter}.

\begin{proposition}[Weighted reverse H\"{o}lder inequality] \label{reverse-holder-bdr} For each $\Lambda>0$ and $M_0 \geq 1$, there is $\epsilon_0 \in (0,1)$ depending only on $n,\Lambda, M_0$ such that the following statement holds. If $\mu \in A_2(\bR^n)$ with $[\mu]_{A_2(\bR^n)} \leq M_0$, we assume $(1+n)(1/l_0'-1)\le 1$, and if $\mu \in A_2(\bR)$ with $[\mu]_{A_2(\bR)} \leq M_0$, we assume $(3+n)(1/l_0'-1)\le 2$. Then for any $q \in [2, 2+ \epsilon_0]$, there exists $C_0 = C_0(n,\Lambda, M_0, q)$ such for any weak solution $u$ of \eqref{Q-2-bdr.eqn}, it holds that
\[
\begin{split}
&  \fint_{Q_{r}^+(\bar z)} |Du|^q \,{\mu}(dz)  \\
&\leq  C_0\left[\left( \fint_{Q_{2r}^+(\bar z)} |Du|^{2} \,{\mu}(dz) \right)^{q/2} +  \fint_{Q_{2r}^+(\bar z)} |\F|^q  \,{\mu}(dz) +  \mathcal{G}_{Q_2^+}(f)^q \fint_{Q_{2r}^+(\bar z)} |f|^{l_0'q} \,{\mu}(dz) \right]
\end{split}
\]
for any $\rho \in (0,1/4)$ and $\bar z \in Q_{3/2}^+$, where
\[
\mathcal{G}_{Q_2^+}(f) = \left(\fint_{Q_2^+} |f(z)|^{2l_0'}\,dz \right)^{\frac{1}{2l_0'} - \frac{1}{2}}.
\]
\end{proposition}
\section{Lipschitz estimates} \label{Lipschits-Section}
In this section we prove several Lipschitz estimates for weak solutions of homogeneous parabolic equations with coefficients that are independent of $(t,x')$, but singular and degenerate in $x_n$-variable. The results will be used later to prove the main theorems. 
The following simple observation will be useful in this section.
\begin{lemma}
Assume that $[\mu]_{A_2(\bR)} \leq M_0$ for some $M_0 >1$ and \eqref{eq11.42c} holds. Then, for any $R\in (0,\infty)$ and a.e. $r\in (0, R]$, it holds that
\begin{equation}
                                \label{eq11.42b}
\left( \int_0^r \mu(s)\,ds \right) \left( \int_0^R \mu(s)\,ds \right) \le CM_0 R^2\mu^2(r).
\end{equation}
\end{lemma}
\begin{proof}
It follows from \eqref{eq11.42} with $y=0$ and $r$ replaced by $R$ that
\begin{align*}
\int_0^r \mu(s)\,ds\cdot \int_0^R \mu(s)\,ds &\le M_0 R^2 \int_0^r \mu(s)\,ds\cdot \left(\int_0^R \mu^{-1}(s)\,ds\right)^{-1}\\
&\le M_0 R^2 \int_0^r \mu(s)\,ds\cdot \left(\int_0^r \mu^{-1}(s)\,ds\right)^{-1}\\
&\le M_0 R^2 \left(\frac 1 r\int_0^r \mu(s)\,ds\right)^{2},
\end{align*}
where we used H\"older's inequality in the last inequality. By the doubling property and \eqref{eq11.42c},
$$
\int_0^r \mu(s)\,ds
\le C \int_{r/2}^r \mu(s)\,ds
\le C \int_{r/2}^{3r/2} \mu(s)\,ds\le Cr\mu(r).
$$
The lemma is proved.
\end{proof}
\subsection{Interior Lipschitz estimates}
Let $\hat{z} = (\hat{t}, \hat{x}) \in \bR^{n+1}$ be fixed, where $\hat{x} = (\hat{x}', \hat{x}_n) \in \bR\times \bR^{n-1}$.  Also, let $\bar{\A} : (\hat{x}_n-3 , \hat{x}_n + 3) \rightarrow \bR^{n \times n}$ be measurable.
We assume that $a_0: (\hat{x}_n-3 , \hat{x}_n + 3) \rightarrow \bR$ is  measurable and satisfies
\begin{equation} \label{a-c-cond}
\Lambda^{-1} \leq a_0(x_n) \leq \Lambda, \quad x_n \in (\hat{x}_n-3 , \hat{x}_n + 3).
\end{equation}
We investigate the regularity for weak solutions of
\begin{equation} \label{inter-eq3.32}
\mu(x_n) a_0(x_n) u_t - \textup{div} [\mu(x_n) \bar{\A}(x_n) \nabla u] + \lambda \mu(x_n) u= 0 \quad \text{in}\,\, Q_3(\hat{z}).
\end{equation}

We begin with the following weighted energy estimates of Caccioppoli type for weak solutions of \eqref{inter-eq3.32}. The proof of this lemma is similar to that of Lemma \ref{B-Cacioppolli} below, and thus skipped.
\begin{lemma} \label{inter-Cacioppolli}  Assume that $\mu: \bR \rightarrow \bR_+$ is locally integrable. 
Then, there exists $C = C(n,\Lambda) >0$ such that for any $\lambda \geq 0$ and for any weak solution $u$ of \eqref{inter-eq3.32}, it holds that
\begin{equation}  \label{eq11.58}
\sup_{\tau \in \Gamma_2(\hat{t})}\int_{D_2(\hat{x}) } u^2(\tau,x) \,\mu(dx)
+\int_{Q_2(\hat{z})}\Big[  |Du|^2 + \lambda |u|^2 \Big] \,{\mu}(dz)  \leq C \int_{Q_3(\hat{z})} u^2 \,{\mu}(dz).
\end{equation}
Moreover,  for any $k, j  \in \mathbb{N} \cup \{0\}$, there is $C =  C(n,k, j, \Lambda)$ such that
\begin{equation*} 
\int_{Q_2(\hat{z}) }|\partial_t^{j+1}u |^2 \,{\mu}(dz) + \int_{Q_2(\hat{z})} |DD_{x'}^k \partial_t^ju |^2 \,{\mu}(dz) \le C\int_{Q_3 (\hat{z})} \Big[ |Du|^2 + \lambda |u|^2\Big] \,{\mu}(dz).
\end{equation*}
\end{lemma}

The following interior Lipschitz estimates for weak solutions of \eqref{inter-eq3.32} is the main result of the subsection.
\begin{proposition} \label{inter-Lipschitz-est} For each $\Lambda >0, M_0 \geq 1$ and $K_0 \geq 1$, there exists a constant $C = C(\Lambda, M_0, K_0)$ such that the following statement holds. Suppose that $\mu \in A_2(\bR)$ with $[\mu]_{A_2(\bR)} \leq M_0$, and 
\begin{equation} \label{extra-radius-1}
\fint_{s_0 -1}^{s_0 +1} \mu(s) \,ds \leq K_0 \mu(s_0) \quad \text{for a .e.} \quad s_0 \in (\hat{x}_n - 1, \hat{x}_n + 1).
\end{equation}
Then, for any $\lambda \geq 0$ and for any weak solution $u$ of \eqref{inter-eq3.32}, we have
\begin{equation}
                                    \label{L-infty-u}
\|u\|_{L^\infty(Q_1(\hat{z}))} \leq C(n,\Lambda, M_0)  \left( \fint_{Q_3(\hat{z})} u^2 \,{\mu} (dz) \right)^{1/2}
\end{equation}
and
\begin{equation*}
\| u_t \|_{L^\infty(Q_1)}  + \| D u\|_{L^\infty(Q_1(\hat{z}))} \leq C(n, \Lambda, M_0, K_0) \left( \fint_{Q_3(\hat{z})} \Big[ |Du|^2 + \lambda |u|^2\Big]\,{\mu} (dz) \right)^{1/2}.
\end{equation*}
\end{proposition}
\begin{proof} The proof is divided into two steps. In the first step, we prove the $L^\infty$-estimates for $u, \partial_t u$ and $D_{x'}u$. In the second step, we establish the $L^\infty$-estimate of $D_{n}u$. Without loss of generality, we assume $\hat{z} = 0$. \\ \noindent
{\bf Step I}:  
Noting that from \eqref{L-infty-u}, Lemma \ref{inter-Cacioppolli}, and since $D_{x'} u$ and $u_t$ satisfy the same equation as $u$, it follows that
 \begin{equation} \label{L-infty-D-x'}
 \begin{split}
& \norm{D_{x'} u}_{L^\infty(Q_1)} \leq C \left (\fint_{Q_3} |D_{x'} u|^2 \,{\mu}(dz) \right)^{1/2}, \\
& \norm{u_t}_{L^\infty(Q_1)} \leq C \left (\fint_{Q_3}\Big[ |Du|^2 + {\lambda |u|^2}\Big] \,{\mu}(dz) \right)^{1/2}.
\end{split}
\end{equation}
The proof of \eqref{L-infty-u} uses Lemma \ref{inter-Cacioppolli} and anisotropic Sobolev inequalities, and is identical to that of \eqref{eq11.43} below in the proof of Proposition \ref{B-Lipschitz-est}. We therefore skip it.
\\ \noindent
{\bf Step II}: From \eqref{L-infty-D-x'}, it remains to show that there is a constant  $C = C(n,\Lambda, M_0, K_0)$ such that
\begin{equation} \label{D-n-inter-infity}
\norm{D_n u}_{L^\infty(Q_1)} \leq C \left(\fint_{Q_3} \Big[ |Du|^2 + {\lambda |u|^2}\Big] \,{\mu}(dz)\right)^{1/2}.
\end{equation}
To prove \eqref{D-n-inter-infity}, let us denote
\begin{equation}
                                \label{eq10.23}
U(t,x)=\displaystyle{\sum_{j=1}^n\bar{\A}_{nj}(x_n) D_j u(t,x)}.
\end{equation}
Observe that we have
\begin{equation} \label{eq11.45-inter}
D_n(\mu(x_n) U)=\mu(x_n) \left[ a_0(x_n)u_t  + \lambda u- \sum_{i =1}^{n-1} \sum_{j =1}^n \bar{\A}_{ij}(x_n) D_{ij} u \right].
\end{equation}
Let us fix $z_0 =(z_0', x_{n0}) \in Q_1$. By applying the Sobolev embedding theorem in $z'$, for fixed $z = (z', x_n)\in Q_{1}(z_0) = Q_1'(z_0') \times (-1 + x_{n0}, 1 + x_{n0}) \subset Q_2$,
we have
\begin{align*}
&|U(z',x_n)|+|u_t(z',x_n)|+|DD_{x'}u(z',x_n)|\\
&\le C\|U(\cdot, x_n)\|_{W^{k/2,k}_2(Q_1'(z_0'))}
+C\|u_t(\cdot, x_n)\|_{W^{k/2,k}_2(Q_1'(z_0'))}
+C\|DD_{x'} u(\cdot, x_n)\|_{W^{k/2,k}_2(Q_1'(z_0'))}
\end{align*}
with an even number $k\ge (d+1)/2$. 
Thus, for any fixed $z'\in Q_{1}'(z_0')$, by Lemma \ref{inter-Cacioppolli} and the fact that $Q_{1}(z_0) \subset Q_2$, it follows that
\begin{align}  \label{eq9.36-inter}
 &\int_{x_{n0}-1}^{x_{n0}+ 1} \mu(x_n) \big(|U(z',x_n)|^2+|u_t(z',x_n)|^2+|DD_{x'} u(z',x_n)|^2\big)\, dx_n\nonumber\\
&\le C \int_{x_{n0} -1}^{ x_{n0}+1} \mu(x_n) \big(
\|U(\cdot, x_n)\|^2_{W^{k/2,k}_2(Q_1'(z_0'))}
+\|u_t(\cdot, x_n)\|^2_{W^{k/2,k}_2(Q_1'(z_0'))}\nonumber\\
&\quad \quad +\|DD_{x'} u(\cdot, x_n)\|^2_{W^{k/2,k}_2(Q_1'(z_0'))}\big)\,dx_n\nonumber\\
&\le C \int_{Q_3} \Big[ |Du|^2 + \lambda |u|^2 \Big] \,{\mu}(dz).
\end{align}
From \eqref{eq11.58}, \eqref{eq11.45-inter}, and \eqref{eq9.36-inter}, we get
\begin{equation*}
\int_{x_{n0}-1}^{x_{n0}+ 1} \mu^{-1}(x_n) \big(|D_n(\mu(x_n) U)|^2+|\mu(x_n) U|^2\big)\, dx_n
\le C \int_{Q_3} \Big[ |Du|^2 + {\lambda |u|^2} \Big] \,{\mu}(dz).
\end{equation*}
Thanks  the doubling property of $\mu$ and by H\"older's inequality, for any fixed $z_0 = (z_0', x_{n0}) \in Q_{1}$,
\begin{align*}
&\left(\int_{x_{n0} -1}^{x_{n0} + 1} |D_n(\mu(x_n) U)|+|\mu(x_n)U|\, dx_n\right)^2\\
&\le \left( \int_{x_{n0} - 1}^{ x_{n0} + 1} \mu^{-1}(x_n) \big(|D_n (\mu(x_n) U)|^2+|\mu(x_n) U|^2\big)\, dx_n \right) \left(
\int_{x_{n0} -1}^{x_{n0} + 1} \mu(x_n)\,dx_n \right)\\
&\le C\left( \fint_{Q_3}  \Big[ |Du|^2 + {\lambda |u|^2}\Big] \,{\mu}(dz) \right)
\left(\int_{x_{n0} -1}^{x_{n0} + 1} \mu(x_n)\,dx_n\right)^{2}.
\end{align*}
This estimate implies that
$$
|\mu (x_{n0}) U (z_0)| \le C \left(\fint_{Q_3} \Big[ |Du|^2 + {\lambda |u|^2}\Big]\, {\mu}( dz)\right)^{1/2}\int_{x_{n0} -1}^{x_{n0} + 1} \mu(x_n)\,dx_n,
$$
for a.e. $z_0 = (z_0', x_{n0}) \in Q_1$, which together with \eqref{extra-radius-1} give
$$
\|U\|_{L^\infty(Q_{1})}\le C \left(\fint_{Q_3} \Big[ |Du|^2 + \lambda |u|^2 \Big]\, {\mu}(dz )\right)^{1/2}.
$$
From this and \eqref{L-infty-D-x'}, we conclude \eqref{D-n-inter-infity}. The proof is therefore complete.
\end{proof}
\subsection{Boundary Lipschitz estimates} Let $\bar{\A} : (0,3) \rightarrow \bR^{n \times n}$ be measurable, and $a_0: (0, 3) \rightarrow \bR$ be measurable and satisfy
\[
\Lambda^{-1} \leq a_0(x_n) \leq \Lambda.
\] 
For $\lambda \geq 0$, we investigate some regularity estimates for weak solutions of
\begin{equation} \label{eq3.32}
\left\{
\begin{array}{cccl}
\mu(x_n) a_0(x_n)u_t - \textup{div} [\mu(x_n) \bar{\A}(x_n) \nabla u]  + \lambda \mu(x_n) u & = &0 \\
\displaystyle{\lim_{x_n \rightarrow 0^+} \wei{\mu(x_n) \bar{\A}(x_n) \nabla u, \mathbf{e}_n} } & = & 0
\end{array} \right. \quad \text{in}\,\, Q_3^+.
\end{equation}
We begin our section with the following lemma on weighted energy estimates of Caccioppoli type for weak solutions of \eqref{eq3.32}.
\begin{lemma} \label{B-Cacioppolli}  Assume that $\mu: \bR_+ \rightarrow \bR_+$ is locally integrable. 
Then, there exists $C = C(\Lambda, n) >0$ such that for any weak solution $u$ of \eqref{eq3.32}, it holds that
\begin{equation} \label{eq3.27}
\sup_{\tau \in [-4, 0]}\int_{D_2^+}u^2(\tau,x) \,\mu(dx)
+\int_{Q_2^+} \Big[ |Du|^2  + \lambda |u|^2 \Big]\,{\mu}(dz)  \leq C \int_{Q_3^+} |u|^2 \,{\mu}(dz).
\end{equation}
Moreover,  for any $k, j  \in \mathbb{N} \cup \{0\}$, there is $C =  C(n,k, j,\Lambda)$ such that
\begin{equation}\label{eq3.50}
\int_{Q_2 ^+}|\partial_t^{j+1}u|^2 \,{\mu}(dz) + \int_{Q_2 ^+} |DD_{x'}^k \partial_t^ju|^2 \,{\mu}(dz) \le C\int_{Q_3^+} \Big[ {|Du|^2 + \lambda |u|^2} \Big]\,{\mu}(dz).
\end{equation}
\end{lemma}
\begin{proof}
Let $0 < r <R <1$, and let $z_0=(t_0,x_0)\in \overline{Q_2^+}$, and $\eta$ be a smooth function satisfying $\eta=1$ in $Q_r(z_0)$, $\eta=0$ outside $(t_0-R^2,t_0+R^2)\times D_R(x_0)$,  and
$$
|D\eta|\le C/|R-r|,\quad |\eta_t|\le C/|R-r|^2.
$$
By using the Steklov's average, we can formally multiply the equation by $\eta^2 u$ and integrate in $Q^+_R(z_0)$ to get the energy inequality
\begin{equation}    \label{eq3.27-1}
\begin{split}
& \sup_{t\in [t_0-r^2,t_0]}\int_{D_r^+(x_0)} u^2(t,x) \,\mu(dx)
+\int_{Q_r^+(z_0)}\Big[ |Du|^2 + {\lambda |u|^2}\Big]\,\mu(dz) \\
& \le C(R-r)^{-2}\int_{Q_R^+(z_0)} |u|^2 \,\mu(dz),
\end{split}
\end{equation}
which implies \eqref{eq3.27}.  By using the difference quotient method, we can see that $D_{x'}^k u$ satisfies the same equation with the same boundary condition as $u$. By induction, we deduce from \eqref{eq3.27-1} that for any $k\ge 1$,
\begin{equation} \label{eq3.30-1}
\begin{split}
& \sup_{t\in [t_0-r^2,t_0]}\int_{D_r^+(x_0)}|D_{x'}^ku|^2 \,\mu(dx)
+\int_{Q_r^+(z_0)}\Big[ |DD_{x'}^ku|^2 +{\lambda |D_{x'}^ku|^2}\Big] \,\mu(dz) \\
& \le C\int_{Q_R^+(z_0)}|Du|^2 \,\mu(dz).
\end{split}
\end{equation}
Next, we estimate time derivatives of $u$. Again, by using difference quotient method and Steklov's average, we can formally multiply \eqref{eq3.32} by $\eta^2 u_t$ and integrate in $Q^+_R(z_0)$ to get
\[
\begin{split}
& \int_{Q_R^+(z_0)}a_0(x_n) u_t^2\eta^2 \,\mu(dz) + \lambda \int_{Q_R^+(z_0)}  u u_t \eta^2 \mu(dz) \\
& \quad +\int_{Q_R^+(z_0)}  \bar{\A}_{ij}(x_n)D_ju (2u_t\eta D_i\eta+D_iu_t\eta^2) \,\mu(dz) =0.
\end{split}
\]
By Young's inequality, for any $\varepsilon\in (0,1)$, we have
\begin{equation*}
\begin{split}
& \int_{Q_R^+(z_0)}  u_t^2\eta^2 \,\mu(dz)\\ 
& \le C\varepsilon^{-1}(R-r)^{-2}\int_{Q_R^+(z_0)} |Du|^2 \,\mu(dz)+ C\varepsilon(R-r)^{2}\int_{Q_R^+(z_0)}|Du_t|^2 \,\mu(dz) \\
& \quad +  \lambda (R-r)^{-2} \int_{Q_R^+(z_0)} u^2 \,\mu(dz).
\end{split}
\end{equation*}
Also, since $u_t$ satisfies the same equation with the same boundary condition as $u$, by \eqref{eq3.27-1} with $u_t$ in place of $u$ and with a slightly different cutoff function, we have
\[
\begin{split}
& \int_{Q_r^+(z_0)}u_t^2 \,\mu(dz) \le C\varepsilon^{-1}(R-r)^{-2}\int_{Q_R^+(z_0)} |Du|^2 \,\mu(dz)
+ C\varepsilon\int_{Q_R^+(z_0)}|u_t|^2 \,\mu(dz)  \\
& \quad + \lambda (R-r)^{-2} \int_{Q_R^+(z_0)} u^2 \,\mu(dz).
\end{split}
\]
This and a standard iteration argument give
\begin{equation}
                                        \label{eq3.45-1}
\int_{Q_r^+(z_0)}u_t^2 \,\mu(dz) \le C(R-r)^{-2} \int_{Q_R^+(z_0)}\Big[ |Du|^2  + {\lambda |u|^2} \Big]\,\mu(dz).
\end{equation}
Again, by the difference quotient method,  we can see that $\partial_t^j u$ satisfies the same equation as $u$. Therefore, by applying \eqref{eq3.30-1} and \eqref{eq3.45-1} repeatedly, we obtain for any integers $j,k\ge 0$,
\begin{equation}
                                        \label{eq3.50-1}
\int_{Q_r^+(z_0)}|DD_{x'}^k \partial_t^ju|^2 \,\mu(dz) \le C\int_{Q_R^+(z_0)} \Big[ |Du|^2 + \lambda |u|^2 \Big]\,\mu(dz) ,
\end{equation}
where $C>0$ depends only on $n$, $j$, $k$, $r$, and $R$. For $j\ge 1$, we also have
$$
\int_{Q_r^+(z_0)} |\partial_t^ju|^2 \,\mu(dz) \le C\int_{Q_R^+(z_0)}  \Big[ |Du|^2+\lambda |u|^2\Big] \,\mu(dz).
$$
The estimates \eqref{eq3.45-1}, \eqref{eq3.50-1}, and the last estimates imply \eqref{eq3.50}.  The proof of the lemma is complete.
\end{proof}

The following boundary Lipschitz estimate is similar to Proposition \ref{inter-Lipschitz-est}.
\begin{proposition} \label{B-Lipschitz-est} For each $\Lambda >0, M_0 \geq 1$ and $K_0\geq 1$, there exists a constant $C = C(n,\Lambda, M_0,  K_0)$ such that the following statement holds true. Suppose that $\mu \in A_2$ satisfying $[\mu]_{A_2} \leq M_0$ and \eqref{eq11.42c}. 
Then, for $\lambda \geq 0$, any weak solution $u$ of \eqref{eq3.32} satisfies the following estimates
\begin{align}
                        \label{eq11.43}
\|u\|_{L^\infty(Q_1^+)} &\leq C(n,\Lambda, M_0) \left( \fint_{Q_3^+} u^2 \,{\mu} (dz) \right)^{1/2},\\
                        \label{eq3.18}
\| u_t \|_{L^\infty(Q_1^+)}  + \| D u\|_{L^\infty(Q_2^+)} &\leq C(n,\Lambda, M_0, K_0) \left( \fint_{Q_3^+} \Big[ |Du|^2 + \lambda |u|^2\Big] \,{\mu} (dz) \right)^{1/2}.
\end{align}
\end{proposition}
\begin{proof} The proof is divided into two steps. In the first step, we prove the $L^\infty$-estimates for $u, \partial_t u$ and $D_{x'}u$. In the second step, we establish the $L^\infty$-estimate of $D_{x_n}u$. \\ \noindent
{\bf Step I}:  In this first step, we prove \eqref{eq11.43}.
From \eqref{eq11.43}, the energy estimates in Lemma \ref{B-Cacioppolli},  and the fact that $D_{x'}u$ and $u_t$ satisfy the same equation as $u$, it follows immediately  that
\begin{equation} \label{eq4.29}
\begin{split}
 & \norm{D_{x'} u}_{L^\infty(Q_1^+)} \leq C \left (\fint_{Q_3^+} |D_{x'} u|^2 \,{\mu}(dz) \right)^{1/2}, \\
 & \norm{u_t}_{L^\infty(Q_1^+)} \leq C \left (\fint_{Q_3^+}\Big[|Du|^2 +\lambda |u|^2\Big] \,{\mu}(dz) \right)^{1/2}.
 \end{split}
\end{equation}
The proof of \eqref{eq11.43} relies on an anisotropic weighted Sobolev inequality. By applying the standard Sobolev embedding theorem in $z':=(t,x')$, for fixed $z=(z',x_n)\in Q_{1}^+$,
we have
$$
|D_nu(z',x_n)|\le C\|D_nu(\cdot, x_n)\|_{W^{k/2,k}_2(Q_1')},
\quad |u(z',x_n)|\le C\|u(\cdot, x_n)\|_{W^{k/2,k}_2(Q_1')}
$$
with an even integer $k\ge (n+1)/2$. Thus, by using \eqref{eq3.27} and \eqref{eq3.50} with slight modification, for any fixed $z'\in Q_{1}'$,
\begin{align}  \label{eq4.16}
&\int_{0}^{1} \mu(x_n) \big(|D_n u(z',x_n)|^2+|u(z',x_n)|^2\big)\, dx_n\nonumber\\
&\le C \int_{0}^{1} \mu(x_n) \big(\|D_n u(\cdot, x_n)\|^2_{W^{k/2,k}_2(Q_1')}+\|u(\cdot, x_n)\|^2_{W^{k/2,k}_2(Q_1')}\big)\,dx_n\nonumber\\
&\le C \int_{Q_3^+} \mu(x_n) u^2\, dz.
\end{align}
Thanks to \eqref{eq11.42}, by H\"older's inequality and \eqref{eq4.16},
\begin{align*}
&\left(\int_{0}^{1} |D_n u(z',x_n)|+|u(z',x_n)|\, dx_n\right)^2\\
&\le \int_{0}^{1} \mu(x_n) \big(|D_n u(z',x_n)|^2+|u(z',x_n)|^2\big)\, dx_n\cdot
\int_{0}^{1} \mu(x_n)^{-1}\,dx_n\\
&\le C \int_{Q_3^+} \mu(x_n) u^2\, dz\cdot
\left(\int_{Q_3^+} \mu(x_n)\, dz\right)^{-1},
\end{align*}
which implies \eqref{eq11.43}.\\
\noindent
{\bf Step II}: To complete the proof of the proposition, it remains to prove that there exists a constant $C = C(n,\Lambda, M_0, K_0)$  such that
\begin{equation}
            \label{eq10.32}
\norm{D_n u}_{L^\infty(Q_1^+)} \leq C \left (\fint_{Q_3^+} \Big[ |Du|^2 +  \lambda |u|^2 \Big] \,{\mu}(dz) \right)^{1/2}.
\end{equation}
Recall \eqref{eq10.23} and \eqref{eq11.45-inter}.
Observe that, by the conormal boundary condition, $(\mu(x_n)U)|_{x_n=0}=0$. 
Moreover, by applying the Sobolev embedding theorem in $z'$, for fixed $z = (z', x_n)\in Q_{1}^+ = Q_1' \times (0,1)$,
we have
\begin{align*}
&|U(z',x_n)|+|u_t(z',x_n)|+|DD_{x'}u(z',x_n)|\\
&\le C\|U(\cdot, x_n)\|_{W^{k/2,k}_2(Q_1')}
+C\|u_t(\cdot, x_n)\|_{W^{k/2,k}_2(Q_1')}
+C\|DD_{x'} u(\cdot, x_n)\|_{W^{k/2,k}_2(Q_1')}
\end{align*}
with an even integer $k\ge (d+1)/2$. Thus by Lemma \ref{B-Cacioppolli}, for any fixed $z'\in Q_{1}'$,
\begin{align}  \label{eq9.36}
 &\int_{0}^{1} \mu(x_n) \big(|U(z',x_n)|^2+|u_t(z',x_n)|^2+|DD_{x'} u(z',x_n)|^2\big)\, dx_n\nonumber\\
&\le C \int_{0}^{1} \mu(x_n) \big(
\|U(\cdot, x_n)\|^2_{W^{k/2,k}_2(Q_1')}
+\|u_t(\cdot, x_n)\|^2_{W^{k/2,k}_2(Q_1')}
+\|DD_{x'} u(\cdot, x_n)\|^2_{W^{k/2,k}_2(Q_1')}
\big)\,dx_n\nonumber\\
&\le C \int_{Q_3^+} \Big[|Du|^2 + \lambda |u|^2\Big]\,{\mu}(dz).
\end{align}
From \eqref{eq11.45-inter}, \eqref{eq3.27}, and \eqref{eq9.36}, and the zero boundary condition, we have
\begin{align*}
|\mu(x_n) U(z)|&\le C\int_0^{x_n} \mu(s) \big(|u_t(z',s)| + {\lambda |u|}+|DD_{x'} u(z',s)|\big)\,ds\\
&\le C \left(\int_0^{1} \mu(s) \big(|u_t(z',s)| +  {\lambda |u|} +|DD_{x'} u(z',s)|\big)^2\,ds\right)^{\frac 1 2}
\left(\int_0^{x_n} \mu(s) \,ds\right)^{\frac 1 2}\\
&\le C\left(\fint_{Q_3^+} \Big[ |Du|^2  + {\lambda |u|^2}\Big] \,{\mu}(dz) \right)^{\frac 1 2}\left(\int_0^{x_n} \mu(s) \,ds\right)^{\frac 12}({\mu}(Q_3^+))^{1/2}.
\end{align*}
Therefore, using \eqref{eq11.42b} we get
\begin{equation*} 
|U(z)|\le  C\left(\fint_{Q_3^+}  \Big[ |Du|^2 + \lambda |u|^2 \Big]\,{\mu}(dz)\right)^{\frac 1 2}.
\end{equation*}
From this last estimate, \eqref{eq10.23}, and \eqref{eq4.29}, we obtain \eqref{eq10.32}. The proof is therefore complete.
\end{proof}
\section{Interior $W^{1,q}$-regularity theory and proof of Theorem \ref{interior-theorem}} \label{interior-Section}
In this section we provide the proof of Theorem \ref{interior-theorem}. To this end, we need to establish several decomposition estimates. As already mentioned, our approach is based on perturbation technique using equations with coefficients independent of $(t,x')$. 
We recall that $l_0$ is a number defined in Lemma \ref{weighted-Sobolev-imbed} which depends only on $M_0$ and $n$, and $l_0'$ is defined in \eqref{l-0-prime}.
\subsection{Interior solution decomposition and their estimates}
Let $z_0 = (t_0, x_0) \in \bR^{n+1}$ and $\rho >0$ be fixed. In this subsection  we investigate interior properties for weak solutions $u$ of the equation
\begin{equation} \label{Q-6.eqn-inter}
\mu(x_n) a_0(x_n)u_t - \textup{div}[\mu(x_n) (\A \nabla u - \F)]  = \mu(x_n)f \quad \text{in} \quad Q_{5\rho}(z_0).
\end{equation}

The main result of the section is the following proposition.
\begin{proposition} \label{inter-approxi-propos}
Let $\Lambda>0$, $K_0 \geq 1$, and $M_0 \geq 1$ be given.  Then, there exists $\kappa_0 = \kappa_0(n,\Lambda, M_0) >0$ such that the following statement holds true. Assume that $\mu \in A_2(\bR)$ with $[\mu]_{A_2} \leq M_0$, and $\A: Q_{5\rho}(z_0) \rightarrow \mathbb{R}^{n\times n}$. 
Assume also that
\begin{equation} \label{mu-rho}
\fint_{y -\rho}^{y +\rho} \mu(s)\, ds \leq K_0 \mu(y) \quad \text{for a.e.} \quad y \in (x_{n0} - \rho, x_{n0}  + \rho).
\end{equation}
Suppose that $\F :Q_{5\rho}(z_0) \rightarrow \bR^n$ and $f: Q_{5\rho}(z_0) \rightarrow \bR$ satisfy $|\F| \in L^2(Q_{5\rho}(z_0), {\mu}), |f| \in L^{2l_0'}(Q_{5\rho}(z_0), {\mu})$. Then, for any weak solution $u \in V^{1,2}(Q_{5\rho}(z_0),\mu)$ of \eqref{Q-6.eqn-inter}, we can write
\[
u(t, x) = \tilde{u}(t, x) + w(t, x) \quad \text{in}\,\, Q_{3\rho}(z_0),
\]
where $\tilde{u}$ and $w$ are functions in $V^{1,2}(Q_{3\rho}(z_0),\mu)$ and they satisfy the following estimates
\begin{align} \label{u-tilde-est-inter}
& \sup_{t \in \Gamma_{3\rho}(t_0)} \rho^{-2} \fint_{D_{3\rho}(x_0)} |\tilde{u} |^2 \,\mu(dx) + \fint_{Q_{3\rho}(z_0)} |D \tilde{u}|^2 \,{\mu}(dz)
\leq C(n,\Lambda, M_0)\Bigg[ \fint_{Q_{4\rho}(z_0)} |\F|^2 \,{\mu}(dz)\nonumber\\
&  \quad + \rho^2 \left( \fint_{Q_{4\rho}(z_0)} |f|^{2l_0'} \,{\mu}(dz)  \right)^{1/l_0'} +[ \A^{\#}_{3\rho}(z_0)]^{\kappa_0} \fint_{Q_{4\rho}(z_0)} |D u|^2 \,{\mu}(dz) \Bigg]
\end{align}
and
\begin{align} \label{L-infty-w-inter}
&\|D w\|_{L^\infty(Q_{\rho}(z_0))}^{2}   \leq C(n,\Lambda, M_0, K_0) \left [  \fint_{Q_{4\rho}(z_0)} |Du|^2 \,{\mu}(dz) \right. \nonumber\\
&\quad  \left. \fint_{Q_{4\rho}(z_0)} |\F|^2 \,{\mu}(dz)  + \rho^2 \left( \fint_{Q_{4\rho}(z_0)} |f|^{2l_0'} \,{\mu}(dz)  \right)^{1/l_0'}  \right].
\end{align}
\end{proposition}

The remaining part of the subsection is to prove Proposition \ref{inter-approxi-propos}. We divide the proof into two steps.  In the first step, we compare $u$ with the weak solution $v$ of the corresponding homogeneous equation
\begin{equation} \label{v-Q5-eqn-inter}
\left\{
\begin{array}{cccl}
\mu(x_n) a_0(x_n) v_t - \textup{div}[\mu(x_n) \A \nabla v ]  & = & 0 & \quad \text{in} \quad Q_{4\rho}(z_0), \\
v & = & u & \quad \text{on} \quad \partial_p Q_{4\rho(z_0)},
\end{array} \right.
\end{equation}
where $\partial_p Q_{4\rho(z_0)}$ denotes the parabolic boundary of $Q_{4\rho(z_0)}$. The result of the first step is stated in the following lemma.
\begin{lemma} \label{inter-1}  Let $\Lambda, M_0$ be positive numbers. Assume that $\mu \in A_2(\bR)$ with $[\mu]_{A_2} \leq M_0$, and $\A : Q_{5\rho}(z_0) \rightarrow \mathbb{R}^{n \times n}$. 
Then, for any $\F$ and $f$ such that $|\F| \in L^2(Q_{5\rho}(z_0), \mu)$ and $|f| \in L^{2l_0'}(Q_{5\rho}(z_0), \mu)$, and for any weak solution $u \in V^{1,2}(Q_{5\rho}(z_0),\mu)$ of \eqref{Q-6.eqn-inter}, there exists a weak solution $v \in  V^{1,2}(Q_{4\rho}(z_0),\mu)$ of  \eqref{v-Q5-eqn-inter} which satisfies
\begin{equation} \label{u-v-inter}
\begin{split}
& \sup_{t \in \Gamma_{4\rho}(t_0)} \rho^{-2} \fint_{D_{4\rho}(x_0)} |u -v|^2 \,\mu(dx) + \fint_{Q_{4\rho}(x_0)} |D( u - v)|^2 \,{\mu}(dz)   \\
& \leq C(n,\Lambda,M_0) \left[ \fint_{Q_{4\rho}(z_0)} |\F|^2 \,{\mu}(dz)  + \rho^2\left( \fint_{Q_{4\rho}(z_0)} |f|^{2l_0'} \,{\mu}(dz)  \right)^{1/l_0'}  \right].
\end{split}
\end{equation}
Moreover, there exists $\epsilon_0 >0$ depending only on $\Lambda, M_0$, and $n$ such that for any $\kappa \in [2,2+\epsilon_0]$,
\begin{equation} \label{v-reverse-inter}
\begin{split}
\left( \fint_{Q_{3\rho}(z_0)} |D v|^{\kappa} \,{\mu}(dz) \right)^{\frac{2}{\kappa}} & \leq C(n,\Lambda, M_0, \kappa)\left[\fint_{Q_{4\rho}(z_0)} | D u |^2 \,\mu(dx) \right. \\
& \quad \quad  + \left.  \fint_{Q_{4\rho}(z_0)} |\F|^2 \,{\mu}(dz)  + \rho^2\left( \fint_{Q_{4\rho}(z_0)} |f|^{2l_0'} \,{\mu}(dz)  \right)^{1/l_0'}  \right].
\end{split}
\end{equation}
\end{lemma}
\begin{proof}
Let $\hat{v}$ be a weak solution of
\begin{equation} \label{hat-v-Q5-eqn-inter}
\left\{
\begin{array}{cccl}
\mu(x_n) a_0(x_n) \hat{v}_t - \textup{div}[\mu(x_n) (\A \nabla \hat{v} - \F) ]   & = & \mu(x_n)f & \quad \text{in} \quad Q_{4\rho}(z_0) \\
\hat{v} & = & 0 & \quad \text{on} \quad \partial_p Q_{4\rho(z_0)}.
\end{array} \right.
\end{equation}
The existence of weak solution $\hat{v}$ of \eqref{hat-v-Q5-eqn-inter} can be obtained through the Galerkin approximation method. Then $v:=u-\hat v$ satisfies \eqref{v-Q5-eqn-inter}. The estimate \eqref{u-v-inter} follows from the standard energy estimates. See the proof of Lemma \ref{inter-energy}. On the other hand, the existence of $\epsilon_0$ and \eqref{v-reverse-inter} follows from Lemma \ref{reverse-holder-inter} applied to $v$ and the triangle inequality.
\end{proof}
In the next step, we compare the solution $v$ of \eqref{v-Q5-eqn-inter} with a solution $w$ of the following equation in which the coefficients are independent of $(t,x')$.
\begin{equation} \label{w-Q4-eqn-inter}
\left\{
\begin{array}{cccl}
\mu(x_n) a_0(x_n) w_t - \text{div}[\mu(x_n) \bar{\A}_{Q_{3\rho}'(z_0')}(x_n) \nabla w]  & = & 0 & \quad \text{in} \quad Q_{3\rho}(z_0), \\
w & = &v & \quad \text{on} \quad \partial_p Q_{3\rho}(z_0).
\end{array}
\right.
\end{equation}
\begin{lemma} \label{inter-2}  Let $\Lambda, M_0$ be positive numbers. Then, there exists $\kappa_0=\kappa_0(n,\Lambda,M_0) >0$ such that the following statement holds. Assume that $\mu \in A_2(\bR)$ with $[\mu]_{A_2} \leq M_0$, and $\A : Q_{5\rho}(z_0) \rightarrow \mathbb{R}^{n \times n}$. 
Then, for any $\F$ and $f$ such that $\F \in L^2(Q_{5\rho}(z_0), {\mu})$ and $f \in L^{2l_0'}(Q_{5\rho}(z_0), {\mu})$, and for any weak solution $u \in V^{1,2}(Q_{5\rho}(z_0),\mu)$ of \eqref{Q-6.eqn-inter}, there exists a weak solution $w \in V^{1,2}(Q_{3\rho}(z_0),\mu)$ of \eqref{w-Q4-eqn-inter}
satisfying
\[
\begin{split}
& \sup_{t \in \Gamma_{3\rho}(t_0)}\rho^{-2} \fint_{D_{3\rho}(x_0)} |v -w|^2 \,\mu(dx) + \int_{Q_{3\rho}(z_0)} |D( v -  w)|^2 \,{\mu}(dz)  \\
& \leq C(\Lambda, M_0, n)[\A^{\#}_{3\rho}(z_0)]^{\kappa_0}  \left[\fint_{Q_{4\rho}(z_0)} | D u |^2 {\mu}(dx) +  \fint_{Q_{4\rho}(z_0)} |\F|^2 \,{\mu}(dz)  \right. \\
& \quad \quad \left. + \rho^2\left( \fint_{Q_{4\rho}(z_0)} |f|^{2l_0'} \,{\mu}(dz)  \right)^{1/l_0'}  \right],
\end{split}
\]
where $v$ is defined in Lemma \ref{inter-1}.
\end{lemma}
\begin{proof}
Let $\hat{w}$ be a weak solution of
\begin{equation} \label{hat-w-Q4-eqn-inter}
\left\{
\begin{array}{cccl}
\mu(x_n) a_0(x_n) \hat{w}_t - \text{div}[\mu(x_n) \bar{\A}_{Q_{3\rho}'(z_0')}(x_n) \nabla \hat{w}]  & = & \textup{div}[\mu(x_n)G] & \quad \text{in} \,\, Q_{3\rho}(z_0), \\
\hat{w} & = &0 & \quad \text{on} \,\, \partial_p Q_{3\rho}(z_0),
\end{array}
\right.
\end{equation}
where
\[
G(t,x) = (\A(t,x) - \bar{\A}_{Q_{3\rho}'(z_0')}(x_n)) \nabla v (t,x).
\]
Observe that $G \in L^2(Q_{3\rho}(z_0), {\mu})$. Indeed, let $\kappa \in (2, 2+\epsilon_0)$, where $\epsilon_0$ is defined in Lemma \ref{inter-1}, by the H\"{o}lder's inequality and the ellipticity condition \eqref{ellipticity} for $\A$ we see that
\begin{align} \label{G-inter-est}
\fint_{Q_{3\rho}(z_0)} |G|^2 \,{\mu}(dz) & = \fint_{Q_{3\rho}(z_0)}  |\A - \bar{\A}_{Q_{3\rho}'(z_0')}(x_n)|^2 |D v |^2 \,{\mu}(dz) \nonumber\\
& \leq \left(\fint_{Q_{3\rho}(z_0)} |\A -\bar{\A}_{Q_{3\rho}'(z_0)}|^{\frac{2\kappa}{\kappa-2}} \,{\mu}(dz) \right)^{\frac{\kappa-2}{\kappa}} \left(\fint_{Q_{3\rho}(z_0)} |D v|^{\kappa} \,{\mu}(dz) \right)^{\frac{2}{\kappa}} \nonumber\\
& \leq C(n,\Lambda, M_0)\left(\fint_{Q_{3\rho}(z_0)} |\A -\bar{\A}_{Q_{3\rho}'(z_0)}| \,{\mu}(dz) \right)^{\frac{\kappa-2}{\kappa}} \left[\fint_{Q_{4\rho}(z_0)} | D u |^2 \,{\mu}(dx) \right. \nonumber\\
& \quad \quad  + \left.  \fint_{Q_{4\rho}(z_0)} |\F|^2 \,{\mu}(dz)  + \rho^2\left( \fint_{Q_{4\rho}(z_0)} |f|^{2l_0'} \,{\mu}(dz)  \right)^{1/l_0'}  \right],
\end{align}
where in the last estimate we have used Lemma \ref{inter-1}. From this estimate, the existence of the weak solution $\hat{w}$ of \eqref{hat-w-Q4-eqn-inter} can be achieved by the Galerkin approximation method. Then $w:=v-\hat w$ satisfies \eqref{w-Q4-eqn-inter}. Moreover, it follows from the energy estimates, as in the proof of Lemma \ref{inter-energy}, that
\[
\sup_{\Gamma_{3\rho}(t_0)} \rho^{-2} \fint_{D_{3\rho}(x_0)} |\hat{w}|^2 \,\mu(dx) + \fint_{Q_{3\rho}(z_0)} |D \hat{w}|^2 \,{\mu}(dz)
\leq C(n, \Lambda) \fint_{Q_{3\rho}(z_0)} |G|^2 \,{\mu}(dz).
\]
From this and \eqref{G-inter-est}, the lemma follows.
\end{proof}
\begin{proof}[Proof of Proposition \ref{inter-approxi-propos}] Let us denote $\tilde{u} = u - w$, where $w$ is defined in Lemma \ref{inter-2}. Then, $ u = \tilde{u} + w$. Moreover,  from Lemmas \ref{inter-1} and \ref{inter-2}, the triangle inequality, and the doubling property of $\mu$, the estimate \eqref{u-tilde-est-inter} follows. Again by the triangle inequality and noting that $\A^\#_{3\rho}(z_0) \leq C(n,\Lambda)$, we also obtain
\[
\begin{split}
  \fint_{Q_{3\rho}(z_0)} |D w|^{2} \,{\mu}(dz) & \leq C(\Lambda, M_0, n) \left[\fint_{Q_{4\rho}(z_0)} | D u |^2 \,{\mu}(dz) \right. \\
& \quad \quad  + \left.  \fint_{Q_{4\rho}(z_0)} |\F|^2 \,{\mu}(dz)  + \rho^2\left( \fint_{Q_{4\rho}(z_0)} |f|^{2l_0'} \,{\mu}(dz)  \right)^{1/l_0'}  \right].
\end{split}
\]
From this, the condition \eqref{mu-rho}, and with some suitable scaling argument, \eqref{L-infty-w-inter} follows from Proposition \ref{inter-Lipschitz-est}. The proof is then complete.
\end{proof}
\begin{proof}[Proof of Theorem \ref{interior-theorem} ] Theorem \ref{interior-theorem} follows  from Proposition \ref{inter-approxi-propos} and level set estimates. The proof is exactly the same as that of Theorem \ref{bdr-reg-est} in the next section. We therefore skip it.
\end{proof}
\section{Boundary $W^{1,q}$-regularity theory and proof of Theorem \ref{bdr-reg-est}} \label{boundary-section}
This section gives the proof of Theorem \ref{bdr-reg-est}.  As in Section \ref{interior-Section}, we apply the freezing coefficient technique. We need to establish several results on solution decompositions near the boundary of the considered domains. Recall that $l_0$ is a number defined in Lemma \ref{weighted-Sobolev-imbed} which depends only on $M_0$, $n$, and $l_0'$ is defined in \eqref{l-0-prime}.

\subsection{Boundary solution decomposition and their estimates}
Let $\rho >0$ be fixed and we study the equation
\begin{equation} \label{B-Q-2.eqn}
\left\{
\begin{array}{cccl}
\mu(x_n)  a_0(x_n) u_t - \textup{div}[\mu(x_n) (\mathbf{A}  \nabla u - \F)] & = & \mu(x_n)f \\
\displaystyle{\lim_{x_n \rightarrow 0^+}\wei{\mu(x_n)(\A \nabla u - \F), \mathbf{e}_n}} & = &0
\end{array} \right.\quad  \text{in}\,\,   Q_{5\rho}^+.
\end{equation}
We denote $\A^\#_{3\rho}=\A^\#_{3\rho}(0)$.
The following proposition is the main result of this subsection.
\begin{proposition} \label{B-approxi-propos} Let $\Lambda>0, K_0 \geq 1, M_0 \geq 1$ be given.  Then, there exists $\kappa_0 = \kappa_0(n,\Lambda, M_0) >0$ such that the following statement holds true. Assume that $\mu \in A_2(\bR)$ with $[\mu]_{A_2} \leq M_0$ and \eqref{eq11.42c} holds. Assume also that $\A: Q_{5\rho}^+ \rightarrow \mathbb{R}^{n\times n}$. 
Suppose that $|\F| :Q_{5\rho}^+ \rightarrow \bR^n$ and $|f|: Q_{5\rho}^+ \rightarrow \bR$ satisfy $\F \in L^2(Q_{5\rho}^+, {\mu})$ and $f \in L^{2l_0'}(Q_{5\rho}^+, {\mu})$. Then, for any weak solution $u \in V^{1,2}(Q_{5\rho}^+,\mu)$ of \eqref{B-Q-2.eqn}, we can write
\[
u(t, x) = \tilde{u}(t, x) + w(t, x) \quad  \text{in}\,\, Q_{3\rho}^+,
\]
where $\tilde{u}$ and $w$ are functions in $V^{1,2}(Q_{3\rho}^+,\mu)$ and they satisfy the following estimates
\begin{equation*} 
\begin{split}
& \sup_{t \in \Gamma_{3\rho}} \rho^{-2} \fint_{D_{3\rho}^+} |\tilde{u} |^2 \,\mu(dx) + \fint_{Q_{3\rho}^+} |D \tilde{u}|^2 \,{\mu}(dz)  \\
& \leq C(n,\Lambda, M_0)\left[ \fint_{Q_{4\rho}^+} |\F|^2 \,{\mu}(dz)  + \rho^2 \left( \fint_{Q_{4\rho}^+} |f|^{2l_0'} \,{\mu}(dz)  \right)^{1/l_0'}  + [ \A^{\#}_{3\rho}]^{\kappa_0} \fint_{Q_{4\rho}^+} |D u|^2 \,{\mu}(dz) \right],
\end{split}
\end{equation*}
and
\begin{equation*} 
\begin{split}
&\|D w\|_{L^\infty(Q_{\rho}^+)}^{2} \\
& \leq C(n,\Lambda, M_0, K_0) \left[  \fint_{Q_{4\rho}^+} |D u|^2 \,{\mu}(dz)  + \fint_{Q_{4\rho}^+} |\F|^2 \,{\mu}(dz)  + \rho^2 \left( \fint_{Q_{4\rho}^+} |f|^{2l_0'} \,{\mu}(dz)  \right)^{1/l_0'}  \right].
\end{split}
\end{equation*}
\end{proposition}

The rest of the section is to prove Proposition \ref{B-approxi-propos}. Similar to the proof of Proposition \ref{inter-approxi-propos}, we divide the proof into two steps. In the first step, we compare the solution $u$ with the solution $v$ of the homogeneous equation
\begin{equation} \label{v-Q32-eqn-B}
\left\{
\begin{array}{cccl}
\mu(x_n) a_0(x_n) v_t - \textup{div}[\mu(x_n) \mathbf{A}  \nabla v ] & = & 0 & \quad  \text{in}\,\,   Q_{4\rho}^+, \\
\displaystyle{\lim_{x_n \rightarrow 0^+}\wei{\mu(x_n)\A \nabla v, \mathbf{e}_n}} & = &0 &\quad \text{in}\,\, Q_{4\rho}^+, \\
 v & = & u & \quad \text{on} \,\, \partial_p Q_{4\rho}^+ \cap \{x_n>0\}.
\end{array} \right.
\end{equation}
Here $\partial_p Q_{4\rho}^+$ denotes the parabolic boundary of $Q_{4\rho}^+$. Also, by weak solution of \eqref{v-Q32-eqn-B}, we mean that $v \in V^{1,2}(Q_{4\rho},\mu)$ satisfying
$v -u =0$ on $\partial_p Q_{4\rho}^+ \cap \{x_n>0\}$ in the sense of trace, and
\[
-\int_{Q_{4\rho}^+} \mu(x_n)  a_0(x_n) v \varphi_t \, dz+ \int_{Q_{4\rho}^+} \wei{\mu(x_n) \A \nabla v, \nabla \varphi}\, dz =0, \quad \forall \ \varphi \in C_0^\infty(Q_{4\rho}).
\]
Our result is stated  and proved in the following lemma.
\begin{lemma} \label{B-step-1-approx} Assume that the conditions in Proposition \ref{B-approxi-propos} hold. Then, for any weak solution $u$ of \eqref{B-Q-2.eqn}, there exists a weak solution $v$ of \eqref{v-Q32-eqn-B} satisfying
\[
\begin{split}
& \sup_{t \in \Gamma_{4\rho}} \rho^{-2} \fint_{D_{4\rho}^+} |u -v|^2 \,\mu(dx)  + \fint_{Q_{4\rho}^+} |D(u-v)|^2 \,{\mu}(dz) \\
& \leq C(\Lambda, M_0, n) \left[ \fint_{Q_{4\rho}^+} |\F|^2 \,{\mu}(dz)  + \rho^2\left( \fint_{Q_{4\rho}^+} |f|^{2l_0'}  \,{\mu}(dz) \right)^{1/l_0'}  \right].
\end{split}
\]
Moreover, there exists $\epsilon_0 >0$ depending only on $\Lambda, M_0$, and $n$ such that for any $\kappa \in [2,2+\epsilon_0]$,
\begin{align*} 
&\left( \fint_{Q_{3\rho}^+} |D v|^{\kappa} \,{\mu}(dz) \right)^{\frac{2}{\kappa}}\\
&\leq C(\Lambda, M_0, n, \kappa)\left[\fint_{Q_{4\rho}^+} | D u |^2 \,\mu(dx) \right. + \left.  \fint_{Q_{4\rho}^+} |\F|^2 \,{\mu}(dz)  + \rho^2\left( \fint_{Q_{4\rho}^+} |f|^{2l_0'} \,{\mu}(dz)  \right)^{1/l_0'}  \right].
\end{align*}
\end{lemma}
\begin{proof}
The proof is similar to that of Lemma \ref{inter-1} by using Proposition \ref{reverse-holder-bdr} instead of Proposition \ref{reverse-holder-inter}.
\end{proof}
In the next step, we compare the solution $v$ of \eqref{v-Q32-eqn-B} with a solution $w$ of the following equation in which the coefficient $\A$ is frozen in  the $(t,x')$-variables.
\begin{equation} \label{w-Q4-eqn-B}
\left\{
\begin{array}{cccl}
\mu(x_n) a_0(x_n) w_t - \text{div}[\mu(x_n) \bar{\A}_{Q_{3\rho}'}(x_n) \nabla w] & = & 0 & \quad \text{in} \quad Q_{3\rho}^+, \\
\displaystyle{\lim_{x_n \rightarrow 0^+} \wei{\mu(x_n) \bar{\A}_{Q_{3\rho}'^+}(x_n) \nabla w, \mathbf{e}_n} }& = & 0 & \quad \text{in} \quad Q_{3\rho}^+, \\
w & = &v & \quad \text{on} \quad \partial_p Q_{3\rho}^+ \cap \{x_n>0\}.
\end{array} \right.
\end{equation}
\begin{lemma} \label{B-2}  Assume that the conditions in Proposition \ref{B-approxi-propos} hold true. Then, there exist a constant $\kappa_0>0$ depending only on $\Lambda, M_0, K_0$ and a weak solution $w \in  V^{1,2}(Q_{3\rho},\mu)$ of \eqref{w-Q4-eqn-B}, which satisfies
\[
\begin{split}
& \sup_{t \in \Gamma_{3\rho}}\rho^{-2} \fint_{D_{3\rho}^+} |w -v|^2 \,\mu(dx) + \int_{Q_{3\rho}^+} |D(w-v)|^2 \,{\mu}(dz)  \\
& \leq C(\Lambda, M_0, n)[\A^{\#}_{3\rho}]^{\kappa_0} \left[\fint_{Q_{4\rho}^+} | D u |^2 \,{\mu}(dz) +  \fint_{Q_{4\rho}^+} |\F|^2 \,{\mu}(dz)  + \rho^2\left( \fint_{Q_{4\rho}^+} |f|^{2l_0'} \,{\mu}(dz)  \right)^{1/l_0'}  \right],
\end{split}
\]
where $v$ is defined in Lemma \ref{B-step-1-approx}. 
\end{lemma}
\begin{proof}
The proof is similar to that of Lemma \ref{inter-2}, and thus omitted.
\end{proof}

From Lemma \ref{B-step-1-approx} and Lemma \ref{B-2}, we are ready to complete the proof of Proposition \ref{B-approxi-propos}.
\begin{proof}[Proof of Proposition \ref{B-approxi-propos}] Let $\tilde{u} = u-w \in Q_{3\rho}^+$, where $w$ is defined in Lemma \ref{B-2}. 
The estimates in Proposition \ref{B-approxi-propos} follows from Lemmas \ref{B-step-1-approx} and \ref{B-2}, Proposition \ref{B-Lipschitz-est}, the triangle inequality, and the doubling property of $\mu$.
\end{proof}
\subsection{Global boundary solution decomposition and their estimates} Recall that for each $z_0 \in \bR \times \overline{\bR^{n}_+}$, $Q_{\rho}^+(z_0) = Q_\rho(z_0) \cap (\bR \times \bR^{n}_+)$ and $D_\rho^+(x) = D_\rho(x) \cap \bR^{n}_+$. 
We study the equation
\begin{equation} \label{main-Q-R-eqn}
\left\{
\begin{array}{cccl}
\mu(x_n) a_0(x_n)u_t - \textup{div}[\mu(x_n)(\A \nabla u - \F)] & = & \mu(x_n)f  \\
\displaystyle{\lim_{x_n \rightarrow 0^+}} \wei{\mu(x_n)(\A \nabla u - \F), \mathbf{e}_n } & = & 0
\end{array}
\right.\quad \text{in} \quad Q_{2}^+.
\end{equation}
We define
\[
[[\A]]_{\textup{BMO}(Q_{3/2}^+, \mu)}  = \sup_{\rho \in (0, R_0)} \sup_{z_0 = (z_0', x_{n0}) \in Q_{3/2}^+}  \fint_{Q_{\rho}^+(z_0)} |\A - \bar{\A}_{Q_{\rho}'(z_0')}(x_n)| \,{\mu}(dz).
\]

The following result is a corollary of Propositions \ref{inter-approxi-propos} and \ref{B-approxi-propos}.

\begin{proposition} \label{G-approx-propos}
Let $\Lambda>0, K_0 >0, M_0 >1$ be given.  Then, there exists $\kappa_0 = \kappa_0(n,\Lambda, M_0) >0$ such that the following statement holds true. Assume that $\mu \in A_2(\bR)$ with $[\mu]_{A_2} \leq M_0$ and \eqref{eq11.42c} holds. Assume also that $\A: Q_{2}^+ \rightarrow \mathbb{R}^{n\times n}$. 
Suppose that $\F :Q_{2}^+ \rightarrow \bR^n$ and $f: Q_{2}^+ \rightarrow \bR$ satisfy $\F \in L^2(Q_{2}^+, {\mu})$ and $f \in L^{2l_0'}(Q_{2}^+, {\mu})$. Then, for any $z_0 = (t_0, x_0) \in Q_{3/2}^+$, $\rho \in (0, R_0/30)$, and for a weak solution $u \in V^{1,2}(Q_2^+,\mu)$ of \eqref{main-Q-R-eqn}, we can write
\[
u(t, x) = \tilde{u}(t, x) + w(t, x) \quad  \text{in}\,\, Q_{3\rho}^+(z_0),
\]
where $\tilde{u}$ and $w$ are functions in $V^{1,2}(Q_{3\rho}^+(z_0),\mu)$ and satisfy
\begin{align}  \label{B-u-tilde-est-inter}
& \sup_{t \in \Gamma_{3\rho}(t_0)} \rho^{-2} \fint_{D_{3\rho}^+(x_0)} |\tilde{u} |^2 \,\mu(dx) + \fint_{Q_{3\rho}^+(z_0)} |D \tilde{u}|^2 \,{\mu}(dz)
 \leq C(\Lambda, M_0, n)\Bigg[\fint_{Q_{30\rho}^+(z_0)} |\F|^2 \,{\mu}(dz)   \nonumber\\
&\quad + \rho^2 \left( \fint_{Q_{30\rho}^+(z_0)} |f|^{2l_0'} \,{\mu}(dz)  \right)^{1/l_0'}+ [[\A]]_{\textup{BMO}(Q_{3/2}^+, \mu)}^{\kappa_0} \fint_{Q_{30\rho}^+(z_0)} |D u|^2 \,{\mu}(dz) \Bigg]
\end{align}
and
\begin{align} \label{D-L-infty-w-inter}
&\|D w\|_{L^\infty(Q_{\rho}^+(z_0))}^{2} \leq C(\Lambda, M_0, K_0, n) \Bigg[  \fint_{Q_{30\rho}^+(z_0)} |D u|^2 \,{\mu}(dz) \nonumber\\
&\quad + \fint_{Q_{30\rho}^+(z_0)} |\F|^2 \,{\mu}(dz)  + \rho^2 \left( \fint_{Q_{30\rho}^+(z_0)} |f|^{2l_0'} \,{\mu}(dz)  \right)^{1/l_0'}  \Bigg].
\end{align}
\end{proposition}
\begin{proof}
We fix $x_0 = (x_0', x_{n0}) \in D_{3/2}^+$ and consider the following two cases depending on $x_{n0}$.\\ \
\noindent
{\bf Case I:} $x_{n0} \geq 5\rho$. In this case, $Q_{5\rho}(z_0) \subset Q_{2}^+$. Therefore, Proposition \ref{G-approx-propos} follows immediately from Proposition \ref{inter-approxi-propos} and the doubling property of $\mu$.\\ \
\noindent
{\bf Case II:} $x_{n0} < 5\rho$. In this case let us denote $\hat{x} = (x_0', 0)$ and note that $\hat{z} = (t_0, \hat{x}) \in Q_{5\rho}(z_0) \cap \{x_n=0\}$. Observe that $Q_{5\rho}^+(z_0) \subset Q_{30\rho}^+(\hat{z}) \subset Q_{2}^+$. Then, by applying Proposition \ref{B-approxi-propos} with $\rho$ replaced by $6\rho$ and $Q_{5\rho}^+$ replaced by $Q_{30\rho}^+(\hat{z})$, we see that there exist $\hat{u}$ and $w$ defined in $Q_{18\rho}^+(\hat{z})$ satisfying $u = \hat{u} +w$ in $Q_{18\rho}^+(\hat{z})$, and
\begin{align}  \label{G-u-tilde-est-inter}
& \sup_{t \in \Gamma_{18\rho}(t_0)} \rho^{-2} \fint_{D_{18\rho}^+(\hat{x})} |\tilde{u} |^2 \,\mu(dx) + \fint_{Q_{18\rho}^+(\hat{z})} |D \tilde{u}|^2 \,{\mu}(dz)
\leq C(n,\Lambda, M_0)\Bigg[ \fint_{Q_{24\rho}^+(\hat{z})} |\F|^2 \,{\mu}(dz)\nonumber   \\
& + \rho^2 \left( \fint_{Q_{24\rho}^+(\hat{z})} |f|^{2l_0'} \,{\mu}(dz)  \right)^{1/l_0'} + [[\A]]^{\kappa_0}_{\textup{BMO}(Q_{3/2}^+, \mu)} \fint_{Q_{24\rho}^+(\hat{z})} |D u|^2 \,{\mu}(dz) \Bigg],
\end{align}
\begin{align} \label{G-L-infty-w-inter}
&\|D w\|_{L^\infty(Q_{6\rho}^+(\hat{z}))} \leq C(n,\Lambda, M_0, K_0) \Bigg[  \fint_{Q_{24\rho}^+(\hat{z})} |D u|^2 \,{\mu}(dz)\nonumber\\
&\quad + \fint_{Q_{24\rho}^+(\hat{z})} |\F|^2 \,{\mu}(dz)  + \rho^2 \left( \fint_{Q_{24\rho(\hat{z})}^+} |f|^{2l_0'} \,{\mu}(dz)  \right)^{1/l_0'}  \Bigg].
\end{align}
Since $Q_{3\rho}^+(z_0) \subset Q_{18\rho}^+(\hat{z})$, we also have $u = \hat{u} + w$ in $Q_{3\rho}^+(z_0)$. Moreover, since $Q_{3\rho}^+(z_0) \subset Q_{18\rho}^+(\hat{z})$ and $Q_{24\rho}^+(\hat{z}) \subset Q_{30\rho}^+(z_0)$, and the doubling property of $\mu$, we get \eqref{B-u-tilde-est-inter}.
Observe also that since $x_{n0} < 5\rho$, $Q_{\rho}^+(z_0) \subset Q_{6\rho}^+(\hat{z})$. It then follows from this fact, the doubling property of $\mu$, and \eqref{G-L-infty-w-inter} that \eqref{D-L-infty-w-inter} holds.
The proof is therefore complete.
\end{proof}
\subsection{Boundary level sets estimates} 
Let $\Lambda >0, M_0 \geq 1$, and $K_0\geq 1$ be fixed constants. Let $2<\eta < 2 + \epsilon_0< q$, where $\epsilon_0 =\epsilon_0(n,\Lambda, M_0)$ validates both Propositions \ref{reverse-holder-inter} and \ref{reverse-holder-bdr}.  Let us also denote
\begin{equation} \label{F-def-level}
F(t, x) =  |\F(t, x)|   + \cG(f) |f(t, x)|^{l_0'} \quad \text{in}\,\, Q_{2}^+,
\end{equation}
where $l_0'$ is defined in \eqref{l-0-prime}, and $\cG(f)=\cG_{Q_2^+}(f)$ defined in Proposition \ref{reverse-holder-bdr}. 
As we will see, the function $\cG$ plays an essential role in our proof. Observe also that since $q \geq2$, by H\"older's inequality,
\begin{equation} \label{G-f}
\cG(f)\norm{|f|^{l_0'}}_{L^{q}(Q_{2}^+, {\mu})} = \left( \fint_{Q_{2}^+} |f|^{2 l_0'} \,{\mu}(dz)\right)^{\frac{1 - l_0'}{2 l_0'}}  \norm{f}_{L^{l_0'q}(Q_{2}^+,{\mu})}^{l_0'}\le {\mu}(Q_{2}^+)^{\frac{1}{q} - \frac{1}{ql_0'}} \norm{f}_{L^{l_0'q}(Q_{2}^+,{\mu})}.
\end{equation}
Note that from \eqref{l-0-prime}, and by choosing $l_0$ sufficiently close to $1$, we also obtain
\begin{equation} \label{l-zero-condition}
(3+n)\Big( \frac{1}{l_0'} -1\Big ) \leq 2.
\end{equation}
Let $\delta >0$ be a constant to be specified later, and let $\tau_0 >0$ be the number defined by
\begin{equation} \label{tau-zero-def}
\tau_0 = \left( \fint_{Q_{3/2}^+} |D u|^2 \,{\mu}(dz) \right)^{1/2} + \frac{1}{\delta} \left( \fint_{Q_{3/2}^+} |F|^{\eta} \,{\mu}(dz) \right)^{1/\eta}  < \infty.
\end{equation}
For fixed numbers $1\leq \rho \leq 3/2$ and $\tau >0$, we denote the upper-level set of $|D u|$ in $Q_{\rho}^+$ by
\begin{equation*} 
E_{\rho}(\tau) =\Big \{\text{Lebesgue point}\ (t, x) \in Q_{\rho}^+ \ \text{of} \ D u: |D u(t,x)| > \tau \Big \}.
\end{equation*}
The following proposition estimating the upper-level sets of $|D u|$ is the main result of this section.
\begin{proposition} \label{main-level-set} There exist $N_0  = N_0 (n,\Lambda, M_0, K_0)>1$, and $B_0 = B_0(n,M_0) >0$  such that the following statement holds. For every $\epsilon \in (0,1)$, there exists $\delta = \delta(\epsilon, \Lambda, M_0, n)>0$ such that if \eqref{l-zero-condition} holds and $[[\A]]_{\textup{BMO}(Q_{3/2}^+, \mu)} \leq \delta$, we have
\[
{\mu}(E_{\bar{\rho}_1}(N_0\tau)) \leq \epsilon  \left[{\mu}(E_{\bar{\rho}_2}(\tau/4)) + \frac{1}{(\delta \tau)^\eta} \int_{\delta \tau/4} ^\infty s^\eta {\mu}\Big ( \{(t,x) \in Q_2^+: |F(t,x)| >s \}\Big)\frac{ds}{s} \right],
\]
for any weak solution $u$ of \eqref{main-Q-R-eqn}, $1 \leq \bar{\rho}_1 < \bar{\rho}_2 \leq 1+R_0$, and $\tau > B_0  (\bar{\rho}_2-\bar{\rho}_1)^{-\frac{n+3}{2}} \tau_0$.
\end{proposition}
The rest of the section is to prove this proposition. We follow and extend the approach developed in \cite{A-Mingione} and used in \cite{Baroni, Bolegein, Bui-Duong}. For each $\bar{z} \in \overline{Q_{2}^+}$ and $\rho  \in (0,1)$, we define
\begin{equation*}
CZ_\rho(\bar{z}) = \left(\fint_{Q_{\rho}^+(\bar{z})} |D u|^2 \,{\mu}(dz) \right)^{1/2} + \frac{1}{\delta}  \left(\fint_{Q_{\rho}^+(\bar{z})} |F|^{\eta} \,{\mu}(dz) \right)^{1/\eta}.
\end{equation*}
Several lemmas are needed to prove Proposition \ref{tau-zero-def}. Our first lemma is a stopping time argument lemma.
\begin{lemma} \label{stopping-lemma} There exists $B_0 = B_0(n,M_0)>0 $ such that for each $1\leq \bar{\rho}_1 < \bar{\rho}_2 \leq 1+R_0$, $\tau >B_0  (\bar{\rho}_2-\bar{\rho}_1)^{-\frac{n+3}{2}} \tau_0$, and for $\bar{z} \in E_{\bar{\rho}_1}(\tau)$, there is $\rho_{\bar{z}} < \frac{\bar{\rho}_2-\bar{\rho}_1}{200}$ such that
\[
CZ_{\rho_{\bar{z}}}(\bar{z}) = \tau, \quad \text{and} \quad CZ_{\rho}(\bar{z}) < \tau \quad \forall \ \rho \in (\rho_{\bar{z}}, 1/2-R_0).
\]
\end{lemma}
\begin{proof}
Observe that for any $\rho \in (0,1/2-R_0)$ and  $\bar{z} = (\bar{t}, \bar{x}', \bar{x}_n) \in Q_{\bar{\rho}_1}^+$, we have $Q_{\rho}^+(\bar{z}) \subset Q_{3/2}^+$. Moreover, since $\mu \in A_2(\bR)$, it follows that
\begin{equation} \label{mu-compare-R-r}
\frac{{\mu}(Q_{3/2}^+)}{{\mu}(Q_\rho^+(\bar{z}))}  \leq C(n,M_0)\rho^{-(n+3)}.
\end{equation}
From \eqref{mu-compare-R-r} and since $\eta >2$,
\[
\begin{split}
CZ_\rho(\bar{z})  & \leq C(n, M_0) \left [\rho^{-(n+3)/2} \left(\fint_{Q_{3/2}^+} |D u|^2 \,{\mu}(dz) \right)^{1/2} +   \rho^{-(n+3)/\eta}  \frac{1}{\delta}  \left(\fint_{Q_{3/2}^+} |F|^{\eta} \,{\mu}(dz) \right)^{1/\eta} \right] \\
& \leq C(n,M_0) \rho^{-(n+3)/2} \tau_0  \leq C(n,M_0) \left(\frac{200}{\bar{\rho}_2-\bar{\rho}_1} \right)^{(n+3)/2} \tau_0 = B_0 (\bar{\rho}_2-\bar{\rho}_1)^{-\frac{n+3}{2}} \tau_0 < \tau
\end{split}
\]
for any
\[
\rho\in \left[\frac{\bar{\rho}_2-\bar{\rho}_1}{200} , \frac{1}{2} - R_0\right] \quad  \text{and} \quad  \tau > B_0  (\bar{\rho}_2-\bar{\rho}_1)^{-\frac{n+3}{2}} \tau_0.
\]
On the other hand, when $\bar{z} \in E_{\bar{\rho}_1}(\tau)$, by the Lebesgue differentiation theorem, we see that if $\rho$ is sufficiently small, then
\[
CZ_\rho(\bar{z}) > \tau.
\]
Due to the fact the $CZ_\rho(\bar{z})$ is continuous in $\rho$, we can find $\rho_{\bar{z}}$, which is  the largest number in $(0, \frac{\bar{\rho}_2-\bar{\rho}_1}{200})$, such that $CZ_{\rho_{\bar{z}}}(\bar{z}) = \tau$. From this,  the conclusion of the lemma follows.
\end{proof}
\begin{lemma} \label{covering-lemma}
For each $1 \leq \bar{\rho}_1 < \bar{\rho}_2 \leq 1 + R_0$ and $\tau>B_0  (\bar{\rho}_2-\bar{\rho}_1)^{-\frac{n+3}{2}} \tau_0$, there exists a countable, disjoint family of cylinders $\{Q_{\rho_i}^+(z_i)\}_{i \in \I}$ with $\rho_i < (\bar{\rho}_2 -\bar{\rho}_1)/200$ and $z_i \in Q_{\bar{\rho}_1}^+$ such that the following holds
\begin{itemize}
\item[\textup{(i)}] $E_{\bar{\rho}_1}(\tau) \subset \cup_{i=1}^\infty Q_{5\rho_i}^+(z_i)$
\item[\textup{(ii)}] $CZ_{\rho_i}(z_i) = \tau$, and $CZ_{\rho}(z_i) < \tau$ for any $\rho \in (\rho_i, 1/2-R_0)$.
\end{itemize}
Moreover, for each $i \in \I$, the following estimate holds
\begin{align} \label{K-i-est}
&{\mu}(Q_{\rho_i}^+(z_i))  \leq C(n,\Lambda, \eta) \Bigg[ {\mu}\Big(Q_{\rho_i}^+(z_i) \cap E_{\bar{\rho}_2}(\tau/4)\Big) \nonumber\\
&\quad + \frac{1}{(\tau \delta)^\eta} \int_{\tau \delta/4}^\infty  s^\eta {\mu}\Big( \{(t,x) \in Q_{\rho_i}^+(z_i): |F(t,x)| > s\} \Big) \frac{ds}{s} \Bigg].
\end{align}
\end{lemma}
\begin{proof}
The conclusions (i) and (ii) follow directly from Lemma \ref{stopping-lemma} and the Vitali covering lemma. It remains now to prove \eqref{K-i-est}.  Observe that if
\begin{equation} \label{case-1-tau-F}
\frac{1}{\delta^\eta} \fint_{Q_{\rho_i}^+(z_i)} |F(t, x)|^\eta\,{\mu}( dz) \geq \frac{\tau^\eta }{2^\eta},
\end{equation}
then
\[
\begin{split}
{\mu} (Q_{\rho_i}^+(z_i))  &\leq \frac{2^{\eta}}{\tau^\eta \delta^\eta} \int_{Q_{\rho_i}^+(z_i)} |F(t, x)|^\eta \,{\mu} (dz) \\
&= \frac{\eta2^{\eta}}{\tau^\eta \delta^\eta} \int_0^\infty s^\eta {\mu} \Big ( \Big\{ (t,x) \in Q_{\rho_i}^+(z_i): |F(t,x)| > s\Big\}\Big) \frac{ds}{s} \\
& = \frac{\eta 2^{\eta}}{\tau^\eta \delta^\eta} \left[ \int_0^{\delta \tau/4} \cdots + \int_{\delta \tau/4}^\infty \cdots \right] \\
& \leq \frac{ {\mu} ( Q_{\rho_i}^+(z_i))}{2^\eta} + \frac{\eta2^{\eta}}{\tau^\eta \delta^\eta} \int_{\delta \tau/4}^\infty s^\eta {\mu} \Big ( \Big\{ (t,x) \in Q_{\rho_i}^+(z_i): |F(t,x)| > s\Big\}\Big )\frac{ds}{s}.
\end{split}
\]
Hence, \eqref{K-i-est} follows. Otherwise, i.e., \eqref{case-1-tau-F} is false, it follows from the fact that $CZ_{\rho_i}(z_i) =\tau$ that
\[
\fint_{Q_{\rho_i}^+(z_i)} |D u|^2 \,{\mu} ( dz) \geq \frac{\tau^2}{2^2},
\]
and therefore
\[
{\mu} ( Q_{\rho_i}^+(z_i) ) \leq \frac{2^2}{\tau^2} \int_{Q_{\rho_i}^+(z_i)} |D u|^2 \,{\mu}(dz).
\]
Then, observe that since $\rho_i < (\bar{\rho}_2- \bar{\rho}_1)/200$, $Q_{\rho_i}^+(z_i) \subset Q_{\bar{\rho}_2}^+$, and hence
\[
\begin{split}
{\mu} ( Q_{\rho_i}^+(z_i)) & \leq \frac{2^2}{\tau^2} \int_{Q_{\rho_i}^+(z_i) \setminus E_{\bar{\rho}_2}(\tau/4)} |D u|^2 \,{\mu} (dz) + \frac{2^2}{\tau^2} \int_{Q_{\rho_i}^+(z_i) \cap E_{\bar{\rho}_2}(\tau/4)} |D u|^2 \,{\mu} (dz) \\
& \leq \frac{{\mu} ( Q_{\rho_i}^+(z_i))|}{4} +  \frac{2^2}{\tau^2} \int_{Q_{\rho_i}^+(z_i) \cap E_{\bar{\rho}_2}(\tau/4)} |D u|^2 \,{\mu} (dz).
\end{split}
\]
Therefore,
\[
{\mu} ( Q_{\rho_i}^+(z_i)) \leq \frac{16}{3\tau^2} \int_{Q_{\rho_i}^+(z_i) \cap E_{\bar{\rho}_2}(\tau/4)} |D u|^2 \,{\mu} (dz).
\]
This and H\"{o}lder's inequality yield that
\[
\begin{split}
{\mu} (Q_{\rho_i}^+(z_i))
& \leq \frac{6  {\mu} (Q_{\rho_i}^+(z_i))^{\frac{2}{\eta}}}{\tau^2} \left(\fint_{Q_{\rho_i}^+(z_i)} |D u|^{\eta} \,{\mu} (dz) \right)^{\frac{2}{\eta}}  {\mu} (Q_{\rho_i}^+(z_i) \cap E_{\bar{\rho}_2}(\tau/4))^{1 - \frac{2}{\eta}}.
\end{split}
\]
Hence,
\begin{equation} \label{R1}
{\mu} ( Q_{\rho_i}^+(z_i))^{1 - \frac{2}{\eta}} \leq \frac{6}{\tau^2} \left(\fint_{Q_{\rho_i}^+(z_i)} |D u|^{\eta} \,{\mu} (dz) \right)^{\frac{2}{\eta}} {\mu} (Q_{\rho_i}^+(z_i) \cap E_{\bar{\rho}_2}(\tau/4)) ^{1 - \frac{2}{\eta}}.
\end{equation}
On the other hand, as $\mu \in A_2(\bR)$ and \eqref{l-zero-condition} holds, by Lemmas \ref{reverse-holder-inter} and \ref{reverse-holder-bdr}, and \eqref{F-def-level},
\begin{align} \label{R3}
& \left(\fint_{Q_{\rho_i}^+(z_i)} |D u|^{\eta} \,{\mu} (dz) \right)^{\frac{1}{\eta}} \nonumber\\
&\leq C \left[\left(\fint_{Q_{2\rho_i}^+(z_i)} |D u|^2\, \mu(dz) \right)^{1/2} + \left(\fint_{Q_{2\rho_i}^+(z_i)} |\F|^{\eta} \,{\mu} (dz) \right)^{\frac{1}{\eta}}+  \mathcal{G}(f) \left( \fint_{Q_{2\rho_i}^+(z_i)} |f|^{l_0'\eta}  \,{\mu} (dz) \right)^{\frac{1}{\eta}} \right]\nonumber\\
&\leq C  \left[\left(\fint_{Q_{2\rho_i}^+(z_i)} |D u|^2\, \mu( dz) \right)^{1/2} +\left(\fint_{Q_{2\rho_i}^+(z_i)} |F|^{\eta} \,{\mu} (dz) \right)^{\frac{1}{\eta}}\right].
\end{align}
Collecting the estimates \eqref{R1} and \eqref{R3},
we conclude that
\begin{align*}
{\mu} (Q_{\rho_i}^+(z_i))^{1 - \frac{2}{\eta}} &\leq \frac{C(n,\Lambda)}{\tau^2} CZ_{2\rho_i}(z_i)^2 {\mu} (Q_{\rho_i}^+(z_i) \cap E_{\bar{\rho}_2}(\tau/4)) ^{1 - \frac{2}{\eta}} \\
&\leq C(n,\Lambda) {\mu} (Q_{\rho_i}^+(z_i) \cap E_{\bar{\rho}_2}(\tau/4)) ^{1 - \frac{2}{\eta}} .
\end{align*}
This implies
\[
{\mu} (Q_{\rho_i}^+(z_i)) \leq C(n,\Lambda)  {\mu} (Q_{\rho_i}^+(z_i) \cap E_{\bar{\rho}_2}(\tau/4)),
\]
and \eqref{K-i-est} follows. The proof is therefore complete.
\end{proof}
\begin{proof}[Proof of Proposition \ref{main-level-set}]
Fix $\bar{\rho}_1, \bar{\rho}_2$, and $\tau$ as in the statement of Proposition \ref{main-level-set}. Note that for each $i \in \I$, $150 \rho_i \in (\rho_i, \frac{1}{2} -R_0)$. Hence, from (ii) of Lemma \ref{covering-lemma}, it follows that $CZ_{150\rho_i}(z_i) < \tau$. Therefore,
\[
\left( \fint_{Q_{150\rho_i}^+(z_i)} |D u|^2 \,{\mu}(dz) \right)^{1/2} \leq \tau, \quad
\left(\fint_{Q_{150\rho_i}^+(z_i)} |F|^\eta \,{\mu}(dz) \right)^{1/\eta} \leq \delta \tau.
\]
Since $Q_{150\rho_i}^+(z_i) \subset Q_2^+$, it follows that there is some constant $C = C(n,M_0) >1$ such that
\begin{equation*} 
\begin{split}
& 150\rho_i \left( \fint_{Q_{150\rho_i}^+(z_i)} |f|^{2l_0'} \,{\mu}(dz) \right)^{1/(2l_0')} \\
& \leq 150 \rho_i \left( \frac{{\mu}(Q_2^+)}{{\mu}(Q_{150 \rho_i}^+(z_i))} \right)^{ \frac{1}{2l_0'} -\frac{1}{2}} \left( \fint_{Q_2^+} |f|^{2l_0'} \,{\mu}(dz) \right)^{\frac{1}{2l_0'} -\frac{1}{2} }  \left( \fint_{Q_{150\rho_i}^+(z_i)} |f|^{2l_0'} \,{\mu}(dz) \right)^{1/2} \\
& \leq C(n,M_0) \cG(f) \left( \fint_{Q_{150\rho_i}^+(z_i)} |f|^{2l_0'} \,{\mu}(dz) \right)^{1/2},
\end{split}
\end{equation*}
where we used \eqref{mu-compare-R-r} and \eqref{l-zero-condition} as well as the definition of $\mathcal{G}$. 
Thus,
\[
 \left( \fint_{Q_{150\rho_i}^+(z_i)} |\F|^{2} \,{\mu}(dz) \right)^{1/2} + 150 \rho_i \left( \fint_{Q_{150\rho_i}^+(z_i)} |f|^{2l_0'} \,{\mu}(dz) \right)^{1/(2l_0')} \leq C(n,M_0)  \delta \tau.
\]
Therefore, for each $i$, by applying Proposition \ref{G-approx-propos} with $\rho$ replaced by $ 5\rho_i \in (0, R_0/30)$, we can find a function $\hat{u}_i$ and $w_i$  defined on $Q_{15\rho_i}^+(z_i)$ such that $u = \hat{u}_i + w_i$ and
\begin{equation} \label{compare-i-l}
\begin{split}
& \fint_{Q_{15 \rho_i}^+(z_i)} |D \hat{u}_i|^2 \,{\mu}(dz)  \leq C(n,\Lambda, M_0, K_0) \left[ [[\A]]^{\kappa_0}_{\textup{BMO}(Q_{3/2}^+, \mu)} + \delta^2 \right]\tau^2 \\
&\quad\quad \leq C_0(n,\Lambda, M_0, K_0) \delta^{\kappa_0} \tau^2, \\
& \norm{D w_i}_{L^\infty(Q_{5\rho_i}^+(z_i))}   \leq C_0(n,\Lambda, M_0, K_0) \tau.
 \end{split}
\end{equation}
Now, let $N_0 = 4C_0(n,\Lambda, M_0, K_0)$. Observe that from Lemma \ref{covering-lemma},
\[
{\mu}\Big( E_{\bar{\rho}_1}(N_0 \tau)\Big) \leq \sum_{i\in \I} {\mu}\Big( \Big\{(t,x) \in Q_{5\rho_i}^+(z_i): |D u(t,x)| > N_0\tau \Big\} \Big).
\]
From this, \eqref{compare-i-l}, and the Chebyshev inequality,  it follows that
\[
\begin{split}
{\mu}\Big(E_{\bar{\rho}_1}(N_0 \tau) \Big)  & \leq \sum_{i\in \I} {\mu}\Big( \Big\{(t, x) \in Q_{5\rho_i}^+(z_i): |D \hat{u}_i(t, x)  | > \frac{N_0\tau}{2} \Big\} \Big) \\
& \quad \quad \quad +
\sum_{i\in \I} {\mu}\Big( \Big\{(t,x) \in Q_{5\rho_i}^+(z_i): |D w_i(t, x) | > \frac{N_0\tau}{2} \Big\} \Big ) \\
& \leq \sum_{i\in \I} {\mu}\Big( \Big\{(t,x) \in Q_{15\rho_i}^+(z_i): |D \hat{u}_i(t, x) | > \frac{N_0\tau}{2} \Big\} \Big )\\
&\leq \Big (\frac{2}{N_0\tau} \Big)^2\sum_{i\in \I}  \int_{Q_{15\rho_i}^+(z_i)} |D \hat{u}_i|^2  \,{\mu}(dz) \\
& \leq C_0^2 \delta^{\kappa_0}\Big (\frac{2}{N_0} \Big)^2\sum_{i\in \I} {\mu}(Q_{15\rho_i}^+(z_i)) \leq C(n,M_0) \delta^{\kappa_0} \sum_{i\in \I} {\mu}(Q_{\rho_i}^+(z_i)),
\end{split}
\]
where in the last estimate, we used the doubling property of $\mu$. From this, if we choose $\delta$ such that $\epsilon = C(n,M_0) \delta^{\kappa_0}$, we get
the conclusion of the proposition follows directly from \eqref{K-i-est} and the fact that  $\{Q_{\rho_i}^+(z_i)\}_{i \in \I}$ is a disjoint family.
\end{proof}
\subsection{Proof of Theorem \ref{bdr-reg-est}} For given constant $\Lambda >0, M_0 \geq 1$, and $K_0\geq 1$ as in the statement of Theorem \ref{bdr-reg-est}, let $\epsilon >0$ be a number to be determined that depends only on $\Lambda, M_0, K_0$, and $n$. Let $\delta = \delta(n,\epsilon, M_0)>0$ be sufficiently small constant defined in Proposition \ref{main-level-set}.  Assume that all the conditions in Theorem \ref{bdr-reg-est} holds with this $\delta$, and we will prove the estimate \eqref{L-q-bdr-est}. \\
\noindent
{\bf Case I:} We assume that $\lambda =0$.  For each $k \in \bN$, we define
$$
(D u)_k (t, x) = \max\Big\{|D u(t, x)|, k \Big\}.
$$
Note that at this moment, we do not know yet if $|D u|$ is in $L^{q}(Q_{1}^+, \mu)$. However, for each fixed $k$,  as  $(D u)_k$ is bounded, $(D u)_k \in L^{q}(Q_{2}^+, \mu)$. For $\rho \in [1,2]$, we denote
\[
E_{\rho}^k(\tau) = \Big \{ (t,x) \in Q_{\rho}^+ : (D u)_k(t, x) > \tau \Big \}.
\]
By considering the cases $k < N_0 \tau$ and $k \geq N_0 \tau$, we can conclude from the Proposition \ref{main-level-set} that
\begin{equation} \label{level-set-k}
{\mu}\Big (E_{\bar{\rho}_1}^k(N_0\tau ) \Big) \leq \epsilon \left[{\mu}\Big(E_{\bar{\rho}_2}^k(\tau/4)  \Big)+ \frac{1}{(\delta \tau)^\eta} \int_{\delta \tau/4} ^\infty s^\eta {\mu}\Big( \Big \{(t, x) \in Q_{2}^+: |F(t, x)| >s\Big \} \Big)\frac{ds}{s} \right]
\end{equation}
for any  $\tau > \hat{B}_0 \tau_0$ with $\hat{B}_0 = B_0 (\bar{\rho}_2 -\bar{\rho}_2)^{-(n+3)/2}$ and $1 \leq \bar{\rho}_1 < \bar{\rho}_2 \leq 1+R_0$.
Then, note that
\begin{equation} \label{q-finite}
\begin{split}
\norm{(D u)_k}_{L^{q}(Q_{\bar{\rho}_1}^+, {\mu})}& \leq C(N_0, q)  \left( \int_{0}^\infty  \tau^q {\mu} \Big(  \big\{(t, x) \in Q_{\bar{\rho}_1}^+: (D u)_k(t, x) > N_0 \tau \big\}  \Big)\, \frac{d\tau}{\tau} \right)^{1/q} \\
& \leq C \left[\left( \int_{0}^{\hat{B}_0 \tau_0} \cdots \right)^{1/q} + \left (\int_{\hat{B}_0 \tau_0} ^{\infty} \cdots  \right)^{1/q} \right] =  I_1 + I_2.
\end{split}
\end{equation}
Using \eqref{tau-zero-def}, the first term $I_1$ is easily controlled as follows
\begin{equation} \label{I-1.est}
\begin{split}
I_1 & \leq C {\mu}(Q_{2}^+)^{1/q}  B_0 (\bar{\rho}_2 -\bar{\rho}_2)^{-(n+3)/2}\tau_0  \\
&\leq C (\bar{\rho}_2 -\bar{\rho}_1)^{-(n+3)/2}\left[{\mu}(Q_{3/2}^+)^{\frac{1}{q} - \frac{1}{2}} \norm{D u}_{L^2 (Q_{3/2}^+, {\mu})} + \delta^{-1}{\mu}(Q_{2}^+)^{\frac{1}{q} - \frac{1}{\eta}}  \norm{F}_{L^\eta(Q_{2}^+, {\mu})}\right] \\
& \leq C(\bar{\rho}_2 -\bar{\rho}_1)^{-(n+3)/2}\left[{\mu}(Q_{3/2}^+)^{\frac{1}{q} - \frac{1}{2}} \norm{D u}_{L^{2}(Q_{3/2}^+, \mu)} + \delta^{-1} \norm{F}_{L^{q}(Q_{2}^+,\mu)} \right].
\end{split}
\end{equation}
For the term $I_2$, we use \eqref{level-set-k} to control it as
\begin{equation} \label{I-2.est}
\begin{split}
I_2 & \leq C \epsilon^{1/q} \left( \int_{\hat{B}_0 \tau_0}^{\infty}  \tau^q{\mu} \Big ( \big\{(t, x) \in Q_{\bar{\rho}_2}^+: (D u)_k(t, x) > \tau/4 \big\} \Big ) \frac{d\tau}{\tau} \right)^{1/q}  \\
& \quad \quad \quad + C\epsilon^{1/q} \delta^{-\eta/q} \left( \int_{\hat{B}_0 \tau_0} ^{\infty}  \tau^{(q-\eta)} \left\{\int_{\delta s/4}^\infty s^{\eta} {\mu} \Big( \{(t, x) \in Q_{2}^+: F(t, x) > s\}\Big) \frac{ds}{s} \right\}\frac{d\tau}{\tau}\right)^{1/q} \\
& \leq C\epsilon^{1/q} \norm{(D u)_k}_{L^{q}(Q_{\bar{\rho}_2}^+, {\mu})}  \\
& \quad \qquad + C \delta^{-1}\left( \int_{\hat{B}_0 \tau_0}^{\infty}  (\delta \tau)^{q-\eta}
\left\{\int_{\delta \tau/4}^\infty s^{\eta} {\mu} \Big(\Big(\{(t,x) \in Q_{2}^+: F(t,x) > s\Big\} \Big) \frac{ds}{s} \right\} \frac{d\tau}{\tau}\right)^{1/q} \\
& = C\left[ \epsilon^{1/q} \norm{(D u)_k}_{L^{q}(Q_{\bar{\rho}_2}, {\mu})}  + J \right],
\end{split}
\end{equation}
where
\begin{equation*} 
J =  \delta^{-1}\left( \int_{\hat{B}_0 \tau_0} ^{\infty}  (\delta \tau)^{q-\eta} \left\{\int_{\delta \tau/4}^\infty s^{\eta}{\mu}\Big(\Big\{(t,x) \in Q_{2}^+: F(t, x) > s \Big\} \Big) \frac{ds}{s} \right\}\frac{d\tau}{\tau}\right)^{1/q}.
\end{equation*}
By Fubini's theorem, $J$ can be controlled as
\[
\begin{split}
J & \leq C \delta^{-1}\left[\int_0^\infty s^{q-\eta} s^{\eta} {\mu} \Big (\big \{(t, x) \in Q_{2}^+: F(t, x) > s\big \}\Big)  \frac{ds}{s} \right]^{1/q} \\
& = C \delta^{-1}\left[\int_0^\infty s^q {\mu}\Big (\big \{(t, x) \in Q_{2}^+: F(t, x) > s\big \}\Big)  \frac{ds}{s} \right]^{1/q} \\
& = C\delta^{-1}  \norm{F}_{L^{q}(Q_{2}^+, {\mu})}.
\end{split}
\]
This estimate, \eqref{q-finite}, \eqref{I-1.est}, and \eqref{I-2.est} imply that
\begin{align*}
&\norm{(D u)_k}_{L^{q}(Q_{\bar{\rho}_1}^+, {\mu})} \\
&\leq C_2\left[\epsilon^{1/q} \norm{(D u)_k}_{L^{q}(Q_{\bar{\rho}_2})} + (\bar{\rho}_2 -\bar{\rho}_1)^{-(n+3)/2}\left({\mu}(Q_{3/2}^+)^{\frac{1}{q} - \frac{1}{2}}\norm{D u}_{L^{2}(Q_{3/2}^+, {\mu})} + \delta^{-1}\norm{F}_{L^{q}(Q_{2}^+, {\mu})} \right)  \right],
\end{align*}
for some constant $C_2$ depending only on $\Lambda$, $M_0$, $K_0$, $n$, and $q$. From this, and by taking $\epsilon$ sufficiently small such that $C_2\epsilon^{1/q} \leq 1/2$, we obtain
\begin{align*}
&\norm{(D u)_k}_{L^{q}(Q_{\bar{\rho}_1}^+, {\mu})} \\
&\leq \frac{1}{2} \norm{(D u)_k}_{L^{q}(Q_{\bar{\rho}_2})} + C (\bar{\rho}_2 -\bar{\rho}_1)^{-(n+3)/2}\left({\mu}(Q_{3/2}^+)^{\frac{1}{q} - \frac{1}{2}}\norm{D u}_{L^{2}(Q_{3/2}^+, {\mu})} + \delta^{-1}\norm{F}_{L^{q}(Q_{2}^+, {\mu})} \right) ,
\end{align*}
for any $1 \leq \bar{\rho}_1 < \bar{\rho}_2 \leq 1 + R_0$.  From this last estimate,  the standard Caccioppoli's estimate and the doubling property of $\mu \in A_2(\bR)$, we have
\begin{align*}
&\norm{(D u)_k}_{L^{q}(Q_{\bar{\rho}_1}^+, {\mu})} \\
&\leq \frac{1}{2} \norm{(D u)_k}_{L^{q}(Q_{\bar{\rho}_2})} + C (\bar{\rho}_2 -\bar{\rho}_1)^{-(n+3)/2}\left({\mu}(Q_{2}^+)^{\frac{1}{q} - \frac{1}{2}}\norm{u}_{L^{2}(Q_{2}^+, {\mu})} + \delta^{-1}\norm{F}_{L^{q}(Q_{2}^+, {\mu})} \right) .
\end{align*}
Then, by a standard iteration lemma (see \cite[Lemma 4.3]{FHL}), we obtain
\[
\norm{(D u)_k}_{L^{q}(Q_{1}^+, {\mu})} \leq C \left({\mu}(Q_{2}^+)^{\frac{1}{q} - \frac{1}{2}}\norm{u}_{L^{2}(Q_{2}^+, {\mu})} + \norm{F}_{L^{q}(Q_{2}^+, {\mu})} \right) .
\]
Finally, by sending $k \rightarrow \infty$, we infer that
\[
\norm{D u}_{L^{q}(Q_{1}^+, {\mu})}  \leq C \left({\mu}(Q_{2}^+)^{\frac{1}{q} - \frac{1}{2}} \norm{u}_{L^{2}(Q_{2}^+, {\mu})} + \norm{F}_{L^{q}(Q_{2}^+, {\mu})} \right).
\]
This estimate and \eqref{G-f} imply 
the desired estimate \eqref{L-q-bdr-est} when $\lambda =0$. \\
\noindent
{\bf Case II:} Now we consider the case $\lambda >0$. We use an idea which was originally due to S. Agmon. For $(t,x)\in \bR^{d+1}_+$ and $y\in \bR$, define
$$
\tilde u(t,x,y)=u(t,x)\sin(\sqrt\lambda y+\pi/4),\quad \tilde f=f_1\sin(\sqrt\lambda y+\pi/4),
$$
$$
\F_i(t,x,y)=\F_i(t,x)\sin(\sqrt\lambda y+\pi/4),\quad \F_{n+1}=f_2(t,x)\cos(\sqrt\lambda y+\pi/4),
$$
and
$$
\tilde \bA_{ij}(t,x,y)=\bA_{ij}(t,x),\quad i,j=1,\ldots,n,
$$
$$
\tilde \bA_{n+1,j}(t,x,y)=\tilde \bA_{j,n+1}(t,x,y)=0,\quad i=1,\ldots,n,\quad \bA_{n+1,n+1}(t,x,y)=1.
$$
It is easily seen that $\tilde\bA$ satisfies the same ellipticity condition \eqref{ellipticity} and $\tilde\bA^\#_r(z,y)\le c_n\delta$ for any $r\in (0,R_0)$ and $(z,y)\in \tilde Q_{3/2}^+$, where $c_n>0$ is a constant depending only on $n$. Moreover, $\tilde u$ is a weak solution of
\begin{equation} \label{eq4.02}
\left\{
\begin{array}{cccl}
\mu(x_n) a_0(x_n) \tilde u_t - \textup{div}[\mu(x_n) (\tilde \bA (t, x,y) \nabla \tilde u - \tilde\F(t,x,y))]& = & \mu(x_n)\tilde f(t, x,y)  \\
\displaystyle{\lim_{x_n \rightarrow 0^+}\wei{\mu(x_n)(\tilde \A(t, x,y) \nabla \tilde u - \tilde\F(t, x,y)), \mathbf{e}_n}} & = &0
\end{array} \right.\quad  \text{in}\,\,   \tilde Q_2^+
\end{equation}
where $\tilde Q_\rho^+:=Q_\rho^+\times (-\rho,\rho)$. By applying {\bf Case I} to  \eqref{eq4.02}, we obtain
\begin{equation} \label{eq4.13}
\begin{split}
\fint_{\tilde Q_1^+} |D \tilde u|^q\, {\mu}(dz)\, dy &  \leq C_1\left( \fint_{\tilde Q_2^+} |\tilde u|^2 \,{\mu}(dz)\, dy \right)^{q/2} \\
& +C_2\left[ \fint_{\tilde Q_2^+} (|\tilde\F|)^q \,{\mu}(dz)\,dy + \left( \fint_{\tilde Q_2^+} |\tilde f|^{ql_0'} \,{\mu}(dz) \,dy\right)^{1/l_0'} \right].
\end{split}
\end{equation}
Note that
$$
D\tilde u=\big(D_x u\sin(\sqrt\lambda y+\pi/4),\sqrt \lambda u\cos(\sqrt\lambda y+\pi/4)\big),
$$
and for any $\lambda\ge 0$,
$$
0<c_{d,q}\le \int_{-1}^{1}|\cos(\sqrt \lambda y+\pi/4)|^q\,dy,
\int_{-1}^{1}|\sin(\sqrt \lambda y+\pi/4)|^q\,dy\le C_{d,q}<\infty.
$$
Therefore, from \eqref{eq4.13} and the definitions of $\tilde f$ and $\tilde \bF$, we easily conclude \eqref{L-q-bdr-est} for general $\lambda\ge 0$.

\section{Proofs of Theorems \ref{thm1.entire} and \ref{thm1.11}} \label{last-section}

We only give the proof of Theorem \ref{thm1.11} since the proof of Theorem \ref{thm1.entire} is similar. From Theorems \ref{interior-theorem} and \ref{bdr-reg-est}, we obtain by using H\"older's inequality that
\[
\begin{split}
&\fint_{Q_1(z_0)} (|D u|+\sqrt \lambda|u|)^q \,{\mu}(dz)   \leq C_1  \fint_{Q_2(z_0)} |u|^q \,{\mu}(dz) \\
& \quad \quad + C_2\left[ \fint_{Q_2(z_0)} (|\F|+|f_2|)^q \,{\mu}(dz) + \left( \fint_{Q_2(z_0)} |f_1|^{ql_0'}\,{\mu}(dz)\right)^{1/l_0'} \right]
\end{split}
\]
for any $z_0=(0,0,x_{n0}) \in \bR^{n+1}_+$ with $x_{n0}\ge 3$, and
\begin{equation*}
\begin{split}
\fint_{Q_1^+} (|D u|+\sqrt \lambda |u|)^q \,{\mu}(dz) &  \leq C_1 \fint_{Q_2^+} |u|^q \,{\mu}(dz) \\
& +C_2\left[ \fint_{Q_2^+} (|\F|+|f_2|)^q \,{\mu}(dz) + \left( \fint_{Q_2^+} |f_1|^{ql_0'} \,{\mu}(dz) \right)^{1/l_0'} \right].
\end{split}
\end{equation*}
Now it follows from a scaling, translation, and a partition of the unity that
\begin{align*}
&\int_{(-\infty,T)\times\bR^n_+} (|D u|+\sqrt \lambda |u|)^q \,{\mu}(dz)  \\
&\leq C\left[ \int_{(-\infty,T)\times\bR^n_+} (|\F|+|f_2|+|u|)^q \,{\mu}(dz) + \left( \int_{(-\infty,T)\times\bR^n_+} |f_1|^{ql_0'} \,{\mu}(dz) \right)^{1/l_0'} \right].
\end{align*}
By taking $\lambda_0$ sufficiently large to absorb the $u$ term on the right-hand side, we obtain \eqref{eq5.25}.
This estimate also gives the uniqueness of weak solutions.

Next, we prove the solvability. In the case when $q=2$, we do not require any regularity condition on $\mathbf{A}$. We use an approximation argument. Let $R>0$ and consider the equation \eqref{eqn-entire} in $(-\infty,T)\times B_R^+$ with 
the conormal boundary condition \eqref{main-eqn} on the flat boundary. 
It follows from the weighted $L^2$-estimates as in Lemma \ref{inter-energy} and the Galerkin's approximation method that there is a weak solution $u_R$ with the estimate uniform with respect to $R$. To get a solution, it remains to let $R\to \infty$ and take a weak limit of $u_R$.

In the case when $q\in (2,\infty)$, for $k=1,2,\ldots$, let $\bF^{(k)}=\bF \chi_{( -k^2, \min\{T, k^2\})\times B_k^+}$. Then $\bF^{(k)}\in L^2((-\infty,T)\times\bR^n_+,\mu)\cap L^q((-\infty,T)\times\bR^n_+,\mu)$, and $\bF^{(k)}\to \bF$ in $L^q((-\infty,T)\times\bR^n_+,\mu)$ as $k\to \infty$. Similarly, we define $\{f_1^{(k)}\}\subset L^{2l_0'}((-\infty,T)\times\bR^n_+,\mu)\cap L^{ql_0'}((-\infty,T)\times\bR^n_+,\mu)$ and $\{f_2^{(k)}\}\subset L^2((-\infty,T)\times\bR^n_+,\mu)\cap L^q((-\infty,T)\times\bR^n_+,\mu)$. Let $u^{(k)}$ be the weak solution of the equation with $\bF^{(k)}$, $ f_1^{(k)}$, and $f_2^{(k)}$ in place of $\bF$, $f_1$, and $f_2$, respectively. By the estimate, we have $u^{(k)}\in W^{1,q}((-\infty,T)\times\bR^n_+,\mu)$. Moreover, by the strong convergence of $\{\bF^{(k)}\}$ and $\{f_2^{(k)}\}$ in $L^q((-\infty,T)\times\bR^n_+,\mu)$, and $\{f_1^{(k)}\}$ in $L^{ql_0'}((-\infty,T)\times\bR^n_+,\mu)$, we infer that
$\{u^{(k)}\}$ is a Cauchy sequence in $W^{1,q}((-\infty,T)\times\bR^n_+,\mu)$. Let $u\in W^{1,q}((-\infty,T)\times\bR^n_+,\mu)$ be its limit. Then it is easily seen that $u$ is a solution to the equation, and satisfies \eqref{eq5.25}. The uniqueness also follows from \eqref{eq5.25}.

Finally, we consider the case when $q\in (1,2)$ by using a duality argument. We first prove the a priori estimate \eqref{eq5.25}. Let $p=q/(q-1)\in (2,\infty)$. For any $\G,g_2\in L^q((-\infty,T)\times\bR^n_+,\mu)$ and $\lambda\ge\lambda_0$, by the proof above, there is a unique solution $v\in W^{1,p}( \bR \times \bR^n_+,\mu)$ to the adjoint problem
\begin{equation*}
\left\{
\begin{array}{cccl}
-\mu(x_n) a_0(x_n) v_t - \textup{div}[\mu(x_n) (\mathbf{A}^{\text{T}}  \nabla v - \G  \chi_{(-\infty,T)})] +\lambda\mu(x_n) v & = & \sqrt\lambda\mu(x_n)g_2\chi_{(-\infty,T)}  \\
\displaystyle{\lim_{x_n \rightarrow 0^+}\wei{\mu(x_n)(\A^{\text{T}} \nabla v - \G  \chi_{(-\infty,T)}), \mathbf{e}_n }} & = &0
\end{array} \right.
\end{equation*}
in $\bR \times \bR^n_+$,
and it satisfies
\begin{equation}  \label{eq5.25b}
\int_{ \bR \times\bR^n_+} (|D v|+\sqrt \lambda |v|)^p  \,{\mu}(dz)  \leq C \int_{(-\infty,T)\times\bR^n_+} (|\G|+|g_2|)^p \,{\mu}(dz).
\end{equation}
 Moreover, it is easily sen that $v=0$ for $t\ge T$. By the equations of $u$ and $v$, we easily get
$$
\int_{(-\infty,T)\times\bR^n_+} (\G\cdot \nabla u+\sqrt\lambda g_2 u)\,{\mu}(dz)
=\int_{(-\infty,T)\times\bR^n_+} (\F\cdot \nabla v+\sqrt\lambda f_2 v)\,{\mu}(dz).
$$
Thus, by H\"older's inequality and \eqref{eq5.25b}, we obtain
\begin{align*}
&\Big|\int_{(-\infty,T)\times\bR^n_+} (\G\cdot \nabla u+\sqrt\lambda g_2 u)\,{\mu}(dz)\Big|\\
&\le \|\F\|_{L^q((-\infty,T)\times\bR^n_+,\mu)}\|Dv\|_{L^p((-\infty,T)\times\bR^n_+,\mu)}
+\sqrt\lambda \|f_2\|_{L^q((-\infty,T)\times\bR^n_+,\mu)}\|v\|_{L^p((-\infty,T)\times\bR^n_+,\mu)}\\
&\le C\big( \|\F\|_{L^q((-\infty,T)\times\bR^n_+,\mu)}+
\|f_2\|_{L^q((-\infty,T)\times\bR^n_+,\mu)}\big)
\big( \|\G\|_{L^p((-\infty,T)\times\bR^n_+,\mu)}+
\|g_2\|_{L^p((-\infty,T)\times\bR^n_+,\mu)}\big).
\end{align*}
Since $\G$ and $g_2$ are arbitrary, \eqref{eq5.25} follows.
The proof of the solvability is more involved in this case. For $i=1,2,\ldots,d$ and $k=1,2,\ldots$, let
$$
\bF^{(k)}_i=\max(-k,\min(k,\bF_i))\chi_{(-k^2, \min\{T, k^2\})\times B_k^+}.
$$
Then $\bF^{(k)}\in L^2((-\infty,T)\times\bR^n_+,\mu)\cap L^q((-\infty,T)\times\bR^n_+,\mu)$, and $\bF^{(k)}\to \bF$ in $L^q((-\infty,T)\times\bR^n_+,\mu)$ as $k\to \infty$. Similarly, we define $\{f_2^{(k)}\}\subset L^2((-\infty,T)\times\bR^n_+,\mu)\cap L^q((-\infty,T)\times\bR^n_+,\mu)$. Let $u^{(k)}\in W^{1,2}((-\infty,T)\times\bR^n_+,\mu)$ be the weak solution of the equation with $\bF^{(k)}$ and $f_2^{(k)}$ in place of $\bF$ and $f_2$, respectively. We claim that $u^{(k)}\in W^{1,q}((-\infty,T)\times\bR^n_+,\mu)$. Assuming this claim, by the strong convergence of $\{\bF^{(k)}\}$ and $\{f_2^{(k)}\}$ in $L^q((-\infty,T)\times\bR^n_+,\mu)$, we infer that
$\{u^{(k)}\}$ is a Cauchy sequence in $W^{1,q}((-\infty,T)\times\bR^n_+,\mu)$. Let $u\in W^{1,q}((-\infty,T)\times\bR^n_+,\mu)$ be its limit. Then it is easily seen that $u$ is a solution to the equation, and satisfies \eqref{eq5.25}.

It remains to prove the claim. Denote $\hat Q_r=( -r^2,\min(T,r^2))\times B_r^+$. Because ${\mu}$ is a doubling measure, we have for any $r>0$,
\begin{equation}
                            \label{eq7.36}
{\mu}(\hat Q_{2r})\le N_0{\mu}(\hat Q_{r}),
\end{equation}
where $N_0$ is independent of $r$. Since $u^{(k)}\in W^{1,2}((-\infty,T)\times\bR^n_+,\mu)$, by H\"older's inequality,
\begin{equation}
                                \label{eq7.38}
\|u^{(k)}\|_{L^q(\hat Q_{2r},\mu)}+\|Du^{(k)}\|_{L^q(\hat Q_{2r} ,\mu)}<\infty.
\end{equation}
For $j\ge 0$, we take a sequence of smooth functions $\eta_j$ such that
$$
\eta_j\equiv 0\quad \text{in}\,\,(-2^{2j}k^2,2^{2j}k^2)\times B_{2^jk},
$$
$$
\eta_j\equiv 1 \quad\text{outside}\,\,(-2^{2(j+1)}k^2,2^{2(j+1)}k^2)\times B_{2^{j+1}k},
$$
and
$$
|D \eta_j|\le C 2^{-j}, \quad |(\eta_j)_t|\le C 2^{-2j},
$$
where $C$ also depends on $k$.
By a simple calculation, we infer that $u^{(k)}\eta_j\in W^{1,2}((-\infty,T)\times\bR^n_+,\mu)$ satisfies
\begin{equation*} 
\left\{
\begin{array}{cccl}
\mu(x_n) a_0(x_n) (u^{(k)}\eta_j)_t - \textup{div}[\mu(x_n) (\mathbf{A}\nabla (u^{(k)}\eta_j) - \F^{(k,j)})] +\lambda\mu(x_n) u^{(k)}\eta_j & = & \mu(x_n)f^{(k,j)} \\
\displaystyle{\lim_{x_n \rightarrow 0^+}\wei{\mu(x_n)(\A \nabla (u^{(k)}\eta_j) - \F^{(k,j)}), \mathbf{e}_n }} & = &0
\end{array} \right.
\end{equation*}
in $(-\infty,T) \times \bR^n_+$,
where
$$
\F^{(k,j)}=u^{(k)}\mathbf{A}\nabla \eta_j,
\quad f^{(k,j)}=u^{(k)}(\eta_j)_t-(\nabla \eta_j,\bA\nabla u^{(k)}).
$$
Applying the estimate \eqref{eq5.25} with $q=2$ to the above equation of $u^{(k)}\eta_j$, we have
\begin{align*}
&\|D(u^{(k)}\eta_j)\|_{L^2((-\infty,T)\times \bR^n_+,\mu)}+\lambda^{1/2} \|u^{(k)}\eta_j\|_{L^2((-\infty,T)\times \bR^n_+,\mu)}\\
&\le C\|\F^{(k,j)}\|_{L^2((-\infty,T)\times \bR^n_+,\mu)}+C\lambda^{-1/2}
\|f^{(k,j)}\|_{L^2((-\infty,T)\times \bR^n_+,\mu)},
\end{align*}
which implies that
\begin{align*}
&\|Du^{(k)}\|_{L^2(\hat Q_{2^{j+2}k}\setminus \hat Q_{2^{j+1}k},\mu)}+\lambda^{1/2} \|u^{(k)}\|_{L^2(\hat Q_{2^{j+2}k}\setminus \hat Q_{2^{j+1}k},\mu)}\\
&\le
C2^{-j}\|u^{(k)}\|_{L^2(\hat Q_{2^{j+1}k}\setminus \hat Q_{2^{j}k,\mu})}+C\lambda^{-1/2}2^{-2j}
\|u^{(k)}\|_{L^2(\hat Q_{2^{j+1}k}\setminus \hat Q_{2^{j}k},\mu)}\\
&\quad +C\lambda^{-1/2}2^{-j}\|Du^{(k)}\|_{L^2(\hat Q_{2^{j+1}k}\setminus \hat Q_{2^{j}k},\mu)}\\
&\le C2^{-j}\big(\|Du^{(k)}\|_{L^2(\hat Q_{2^{j+1}k}\setminus \hat Q_{2^{j}k},\mu)}+\lambda^{1/2} \|u^{(k)}\|_{L^2(\hat Q_{2^{j+1}k}\setminus \hat Q_{2^{j}k},\mu)}\big),
\end{align*}
where $C$ also depends on $\lambda$. By iteration, we get that for each $j\ge 1$,
\begin{align}
                        \label{eq9.01}
&\|Du^{(k)}\|_{L^2(\hat Q_{2^{j+1}k}\setminus \hat Q_{2^{j}k},\mu)}+\lambda^{1/2} \|u^{(k)}\|_{L^2(\hat Q_{2^{j+1}k}\setminus \hat Q_{2^{j}k},\mu)}\nonumber\\
&\le C^j2^{-j(j-1)/2}\big(\|Du^{(k)}\|_{L^2(\hat Q_{2k},\mu)}+\lambda^{1/2} \|u^{(k)}\|_{L^2(\hat Q_{2k},\mu)}\big).
\end{align}
Finally, by H\"older's inequality, \eqref{eq7.36}, and \eqref{eq9.01}, we have
\begin{align*}
&\|Du^{(k)}\|_{L^q(\hat Q_{2^{j+1}k}\setminus \hat Q_{2^{j}k},\mu)}+\lambda^{1/2} \|u^{(k)}\|_{L^q(\hat Q_{2^{j+1}k}\setminus \hat Q_{2^{j}k},\mu)}\\
&\le ({\mu}(\hat Q_{2^{j+1}k}))^{1/q-1/2}\big(\|Du^{(k)}\|_{L^2(\hat Q_{2^{j+1}k}\setminus \hat Q_{2^{j}k},\mu)}+\lambda^{1/2} \|u^{(k)}\|_{L^2(\hat Q_{2^{j+1}k}\setminus \hat Q_{2^{j}k},\mu)}\big)\\
&\le N_0^{j(1/q-1/2)}C^j2^{-j(j-1)/2}({\mu}(\hat Q_{2k}))^{1/q-1/2}\big(\|Du^{(k)}\|_{L^2(\hat Q_{2k},\mu)}+\lambda^{1/2} \|u^{(k)}\|_{L^2(\hat Q_{2k},\mu)}\big),
\end{align*}
which together with \eqref{eq7.38} implies the claim. The theorem is proved.




\end{document}